\documentclass[12pt,reqno]{amsart}
\usepackage{amsmath,amssymb,amsfonts,amscd,latexsym,amsthm,mathrsfs,verbatim}
\usepackage[colorlinks,linkcolor=black,citecolor=black]{hyperref}
\usepackage{hypbmsec}
\usepackage{bm}
\usepackage[margin=1in]{geometry}
\usepackage{centernot}
\newtheorem{lemma}{Lemma}
\newtheorem{theorem}{Theorem}

\theoremstyle{remark}
\newtheorem*{remark}{\bf Remark}

\renewcommand{\Re}{\operatorname{Re}}
\newcommand{\Li}{\operatorname{Li}}

\newcommand{\Z}{\mathbb{Z}}

\newcommand{\R}{\mathbb{R}}
\newcommand{\C}{\mathbb{C}}
\newcommand{\N}{\mathbb{N}}

\newcommand{\e}{\operatorname{e}}
\newcommand{\modu}{\operatorname{mod}}
\newcommand{\real}{\operatorname{Re}}

\usepackage{etoolbox}
\patchcmd{\section}{\scshape}{\bfseries}{}{}
\makeatletter
\renewcommand{\@secnumfont}{\bfseries}
\makeatother

\numberwithin{equation}{section}
\numberwithin{lemma}{section}
\numberwithin{theorem}{section}
\numberwithin{prop}{section}

\begin{document}

\title{Polynomial partition asymptotics}

\author{Alexander Dunn}
\address{Department of Mathematics, University of Illinois, 1409 West Green
Street, Urbana, IL 61801, USA}
\email{ajdunn2@illinois.edu}

\author{Nicolas Robles}
\address{Department of Mathematics, University of Illinois, 1409 West Green
Street, Urbana, IL 61801, USA}
\email{nirobles@illinois.edu}
\address{Wolfram Research Inc, 100 Trade Center Dr, Champaign, IL 61820, USA}
\email{nicolasr@wolfram.com}


\maketitle

\begin{abstract}
Let $f \in \mathbb{Z}[y]$ be a polynomial such that $f(\mathbb{N}) \subseteq \mathbb{N}$, and let $p_{\mathcal{A}_{f}}(n)$ denote number of partitions of $n$ whose parts lie in the set $\mathcal{A}_f:=\{f(n):n \in \mathbb{N}\}$. Under hypotheses on the roots of $f-f(0)$, we use the Hardy--Littlewood circle method, a polylogarithm identity, and the Matsumoto--Weng zeta function  to derive asymptotic formulae for $p_{\mathcal{A}_f}(n)$ as $n$ tends to infinity. This generalises asymptotic formulae for the number of partitions into perfect $d$th powers, established by Vaughan for $d=2$, and Gafni for the case $d \geq 2$, in 2015 and 2016 respectively.
\end{abstract}

\section{Introduction and Preliminaries} \label{intro}
A partition of a positive integer $n$ is a non-decreasing sequence of positive integers whose sum is $n$. Let $\mathcal{A} \subseteq \mathbb{N}$ and $p_{\mathcal{A}}(n)$ denote the number of partitions of $n$ such that each part of the partition is restricted to be an element of $\mathcal{A}$. When $\mathcal{A}:=\mathbb{N}$, we obtain the well studied unrestricted partition function, usually denoted by $p(n)$. Let $f \in \mathbb{Z}[y]$ be a polynomial such that $f(\mathbb{N}) \subseteq \mathbb{N}$. Then we define $p_{\mathcal{A}_{f}}(n)$ to be the number of partitions of $n$ whose parts lie in the set $\mathcal{A}_f:=\{f(n):n \in \mathbb{N}\}$. Under mild hypotheses on $f$, we derive an asymptotic formula for $p_{\mathcal{A}_f}(n)$ using the Hardy--Littlewood circle method and a fine analysis of the Matsumoto--Weng zeta function \cite{MW}. 

In 1918, Hardy and Ramanujan initiated the analytic study of $p(n)$ with the use of the celebrated Hardy--Littlewood circle method \cite{HR}.  They proved 
\begin{equation*}
p(n) \sim \frac{1}{4 \sqrt{3} n} e^{\pi \sqrt{2n/3}} \quad \text{as} \quad n \rightarrow \infty.
\end{equation*}
For fixed $k \geq 2$ they also conjectured an asymptotic formula for the restricted partition function $p_{\mathcal{A}_k}(n)$, where $\mathcal{A}_k$ denotes the set of perfect $k$th powers. Later in 1934, Wright \cite{W} provided proof for Hardy and Ramanujan's conjectured formula concerning $p_{\mathcal{A}_k}(n)$. However, Wright's proof relied heavily on a transformation for the generating function for the sequence $\{p_{\mathcal{A}_k}(n)\}$ that involved generalised Bessel functions. 

Vaughan has recently established a simplified asymptotic formula for $p_{\mathcal{A}_k}(n)$ in the case $k=2$ \cite{V2}. This  was subsequently generalised for all $k \geq 2$ by Gafni \cite{G}. Using the ideas from \cite{V2} and \cite{G}, Berndt, Malik, and Zaharescu in \cite{BMZ} have derived an asymptotic formula for restricted partitions in which each part is a $k$th power in an arithmetic progression. More precisely, for fixed $a_0,b_0,k \in \mathbb{N}$ with $(a_0,b_0)=1$, they give an asymptotic for $p_{\mathcal{A}_k(a_0,b_0)}(n)$, as $n$ tends to infinity, where $\mathcal{A}_k(a_0,b_0):=\{m^k: m \equiv a_0 \modu{b_0}\}$. It is at the end of Berndt, Malik, and Zaharescu's paper \cite{BMZ} that they pose the question of establishing an asymptotic formula for $p_{\mathcal{A}_f}(n)$. To this end, we will follow the implementation of the circle method presented in \cite{BMZ,G,V2}, with some key innovations. The first is a careful analysis of the Matsumoto--Weng zeta function and the application of a polylogarithm identity to extract the main terms of the asymptotic occurring in Theorem \ref{theorem1}. For this see Lemma \ref{1}. The second key innovation is a generalisation of the classical major arc estimate for Waring's problem, see Lemma \ref{sing}.

For other types of formulae for restricted partitions, we refer the reader to \cite{NS1} and \cite{NS2}. Interestingly, Vaughan has obtained an asymptotic formula for the number of partitions into primes \cite{V3}.

We now introduce some notation and preliminaries that will allow us to state Theorem \ref{theorem1}.
Let $d \geq 2$ and suppose
\begin{equation*}
f(y):=\sum_{j=0}^d a_j y^j \in \mathbb{Z}[y]
\end{equation*}
is fixed such that $(a_0,\ldots,a_d)=1$ and 
\begin{equation*}
f(y)-a_0=a_d y \prod_{j=1}^{d-1} (y+\alpha_j)
\end{equation*}
is such that $\alpha_j \in \mathbb{C} \setminus \mathbb{R}_{\leq -1}$. By convention, for $z \in \mathbb{C}$ we let $-\pi<\arg(z) \leq \pi$ and $\boldsymbol{\alpha}:=(\alpha_1,\ldots,\alpha_{d-1},0)$. Note that the greatest common divisor condition imposed above is important because it ensures there are no congruence obstructions to representing an integer $n$ with a partition whose parts are values of $f$.

Evaluations of the Matsumoto--Weng zeta function at integers and residues of its poles naturally appear in our asymptotic formulae for $p_{\mathcal{A}_f}(n)$. We will provide some brief background on this  function. Matsumoto and Weng \cite{MW} introduced the following $r$-tuple zeta function 
\begin{equation} \label{rtuple}
\zeta_r \big((s_1,\ldots,s_r);(\beta_1,\ldots,\beta_r) \big):=\sum_{n=1}^{\infty} \frac{1}{(n+\beta_1)^{s_1} \cdots (n+\beta_r)^{s_r}}
\end{equation}
where the $s_j \in \C$ are complex variables and $\beta_j \in \mathbb{C} \setminus \R_{\leq -1}$ for all $1 \leq j \leq r$. Here
\begin{equation*}
(n+\beta_j)^{s_j}=\exp(-s_j \log (n+\beta_j ))
\end{equation*}
where the branch of the logarithm is fixed as $-\pi< \arg(n+\beta_j) \leq \pi$. This series is clearly well defined and absolutely convergent in the region,
\begin{equation*}
\Re(s_1+\cdots+s_r)>1.
\end{equation*}
By means of the classical Mellin--Barnes integral formula \cite[Eqn. 4]{MW}, $\zeta_r(\cdot,\boldsymbol{\beta})$ has meromorphic continuation to $\mathbb{C}^r$ with respect to the variables $s_1,\ldots,s_r$ when $\boldsymbol{\beta}=(\beta_1,\ldots,\beta_{r-1},0)$. One can see \cite[Prop.~1]{MW} for more details. We will use the one--variable specialisation $s:=s_1=\cdots=s_r$ of \eqref{rtuple}, and its corresponding meromorphic continuation with respect to $s$ when $\beta_r=0$. This one--variable function will be denoted by $\zeta(\cdot,\boldsymbol{\beta})$. By \cite[Prop.~1]{MW}, the function $\zeta(\cdot,\boldsymbol{\beta})$ has a simple pole at $s=1/r$ with residue $1/r$, and at most simples poles at $s=(1-m)/r$ by \cite[Lem.~5]{MW} with $m \in \mathbb{N}$. Let $c_m$ denote the residue of $\zeta(\cdot,\boldsymbol{\beta})$ at such point. The $c_m$ are given by sums of products of multinomial coefficients with the roots $\beta_r$, and examples can be found on \cite[pp.~246--247]{MW}. In fact $\zeta(\cdot,\boldsymbol{\beta})$ is analytic all non-positive integers by \cite[Prop.~2]{MW}, and we have the explicit evaluation
\begin{equation} \label{evalz}
\zeta(0,\boldsymbol{\beta})=\zeta(0)-\frac{1}{r} \sum_{j=1}^{r-1} \beta_j
\end{equation}
where $\zeta(\cdot)$ denotes the Riemann zeta function. There is also an evaluation of $\zeta^{\prime}(0,\boldsymbol{\beta})$ in \cite[Theorem~E]{MW}.

Now we introduce some parameters depending on $n$ that occur in our asymptotic formulae. For each $n \in \mathbb{N}$ sufficiently large (depending only on $f$), let $X,Y \in \mathbb{R}$ be such that 
\begin{equation} \label{X}
n = X\bigg(\frac{1}{{{d^2}}}\zeta \Big( {\frac{{1 + d}}{d}} \Big) \Gamma \Big( {\frac{1}{d}} \Big) \Big(\frac{X}{a_d} \Big)^{1/d}+ \zeta (0,\boldsymbol{\alpha} )+\frac{W_l^{\prime}(0)}{2 \pi i X}  \bigg)
\end{equation}
where $W_l$ for $l=0,1$ (depending on whether $a_0=0$ or $a_0 \neq 0$ respectively) are computable constants defined in Lemma \ref{1} and
\begin{equation} \label{Y}
Y := \frac{{d + 1}}{{2{d^3}}} \Big(\frac{X}{a_d} \Big)^{1/d} \Gamma \Big( {\frac{1}{d}} \Big)\zeta \bigg( {\frac{1}{d} + 1} \Big) + \frac{1}{2}\zeta (0,\boldsymbol{\alpha} )-\frac{W_l^{\prime \prime}(0)}{8 \pi^2 X^2}.
\end{equation}
In the cases we consider, $W_l^{\prime}(0)$ is purely imaginary, so $X \in \mathbb{R}$ is well-defined. Here it is instructive to observe that $X \sim \mathcal{C}_1 n^{\frac{d}{d+1}}$ and $Y \sim \mathcal{C}_2 n^{\frac{1}{d+1}}$ as $n \rightarrow \infty$ for some computable constants $\mathcal{C}_1,\mathcal{C}_2>0$, depending only on $f$. Under certain hypotheses on $f$,  the implication of Theorem \ref{theorem1} will be
\begin{equation*}
p_{\mathcal{A}_f}(n) \sim \mathcal{C}_3 n^{-\frac{2d(1-\zeta(0,\boldsymbol{\alpha}))+1}{2(d+1)}}  \exp \big(\mathcal{C}_4 n^{\frac{1}{d+1}} \big)
\end{equation*}
for some computable constants $\mathcal{C}_3,\mathcal{C}_4>0$, depending only on $f$. Keeping \eqref{evalz} in mind, it is interesting how the roots $-\alpha_1,\ldots,-\alpha_{d-1}$ of $f-a_0$ affect the main terms of the asymptotics. To state our results more precisely let
\begin{equation*}
\mathcal{C}:=\frac{1}{d}  \zeta \left( \frac{1 + d}{d} \right) \Gamma \Big( \frac{1}{d} \Big) \Big(\frac{X}{a_d} \Big)^{1/d}+W(0)
\end{equation*}
where the appropriate $W$ is chosen depending on $a_0$.  

\begin{theorem} \label{theorem1}
Fix $f \in \mathbb{Z}[y]$ as above and suppose $\alpha_j \in \mathbb{R}_{\geq 0}$ for all $1 \leq j \leq d-1$ . Suppose
\begin{equation*}
a_0 \geq 0, \quad a_1=0, \quad \frac{a_{d-1}}{a_d}:=\sum_{j=1}^{d-1} \alpha_j < \frac{d}{2},  \quad \frac{a_0}{a_d}<1
\end{equation*}
and $f$ is non-constant as a function $\modu{p}$ for all primes $p \leq d$.  Let $n,X,Y \in \mathbb{R}$ be as above and $0<R<1$ fixed. Then, for any $1<J<dR$, there exist $w_1,\cdots,w_{J-1} \in \mathbb{R}$ (independent of $n$) such that  
\begin{equation*}
p_{\mathcal{A}_f}(n)=\frac{1}{2 \pi a_d^{\zeta(0,\boldsymbol{\alpha})}} \frac{\exp \Big(\mathcal{C}+\frac{n}{X} \Big)}{X^{1-\zeta(0,\boldsymbol{\alpha})} Y^{\frac{1}{2}}} \Bigg(\sqrt{\pi}+\sum_{q=1}^{J-1} w_q Y^{-q} +O_{f,R}(Y^{-J})  \Bigg)
\end{equation*}
as $n \rightarrow \infty$.
\end{theorem}

Note that the hypotheses placed on $f \in \mathbb{Z}[y]$ in Theorem 1 guarantee that all coefficients of $f$ are non-negative.

In the case $f(x)=x^d$ with $d \geq 2$ fixed, Theorem \ref{theorem1} yields the same main terms occurring in \cite[Theorem~1]{G}.

\section{Auxiliary lemmas} \label{Auxiliary lemmas}
The generating series for $\{p_{\mathcal{A}_f}(n) \}$ is given by 
\begin{equation} \label{par}
\Psi_{f}(z):=\sum_{n=1}^{\infty} p_{\mathcal{A}_f}(n) z^n =  \prod_{n=1}^{\infty} \big(1-z^{f(n)} \big)^{-1}.
\end{equation}
It will be convenient in computations to define the logarithm of the above 
\begin{equation} \label{aux}
\Phi_f(z):=\sum_{j=1}^{\infty} \sum_{n=1}^{\infty} \frac{1}{j} z^{j f(n)}.
\end{equation}
For $X, \Theta \in \mathbb{R}$ with $X \geq 1$, the following quantities will be useful
\begin{equation*}
\rho:=\exp \big({-1/X} \big), \quad \Delta:=(1+4 \pi^2 X^2 \Theta^2)^{-1/2} \quad \text{and} \quad x:=\frac{1-2 \pi i \Theta X }{X}.
\end{equation*}

\subsection{Bound for Matsumoto--Weng zeta function}
In the special case we consider, the Matsumoto--Weng zeta function has a polynomial bound in  bounded vertical strips of the complex plane. This will be useful throughout the proof of Lemma \ref{1}.
\begin{lemma} \label{MWbound}
Suppose $\boldsymbol{\beta}=(\beta_1,\ldots,\beta_{d-1},0)$ where $\beta_j \in \mathbb{R}_{\geq 0}$ for $1 \leq j \leq d-1$.  Then for $s=\sigma+it$ with $|t|>1$ we have 
\begin{equation} \label{realMW}
|\zeta(s,\boldsymbol{\beta})| \ll_{\boldsymbol{\beta},\sigma} |t|^{O_{\boldsymbol{\beta},\sigma}(1)}. 
\end{equation}
\end{lemma}
\begin{proof}
Without loss of generality we may assume there is at least one $1 \leq j \leq d-1$ such that $\beta_j \neq 0$. Otherwise $\zeta(s,\boldsymbol{\beta})=\zeta(ds)$ and we have $\zeta(ds) \ll |t|^{O_d(1)}$. We may permute the $\beta_j$ appropriately to suppose $1 \leq l \leq d-1$ is such that  $\beta_j \neq 0$ for $1 \leq j \leq l$ and $\beta_j=0$ otherwise. Thus we can make the reduction   
\begin{equation} \label{conve}
\zeta \big(s,\boldsymbol{\beta} \big)=\zeta_{l+1} \big(s,\ldots,s,(d-l)s; (\beta_1,\ldots,\beta_l,0) \big).
\end{equation}  
The result follows from now applying \cite[Proposition 1(iii)]{MW} to the right hand side of \eqref{conve}.
\end{proof}

\subsection{Asymptotics for $\Phi_f$}
To obtain asymptotics for $\Phi_f$, our strategy is to apply the Cahen--Mellin tranform \cite[Entry 3.1]{ober} to the summand of $\Phi_f$. We then employ a polylogarithm identity to simplify the resulting contour integral. The rest of the proof of Lemma \ref{1} will largely proceed via residue calculus and controlling the error when we shift lines of integration. The $W_0$ and $W_1$ terms come from the residues of $\zeta(s,\boldsymbol{\alpha})$ at various poles. In what follows let $e(x):=\exp(2 \pi i x)$.

\begin{lemma} \label{1}
Let $f \in \mathbb{Z}[y]$ be as in Theorem \ref{theorem1} and $R,\varepsilon>0$ fixed. Let $0<L<1$ be sufficiently small and fixed depending on $f,R$ and $\varepsilon$. Then for all $X>0$ sufficiently large and $\Theta\in \mathbb{R}$ such that $| \Theta| \leq  X^{\frac{L}{d}-1}$, we have
\begin{multline} \label{auxphi}
\Phi_f \big( \rho e (\Theta) \big)=\frac{a_d^{-1/d}}{d} \Gamma \Big(\frac{1}{d} \Big) \zeta \Big( 1+\frac{1}{d} \Big) \Big(\frac{X}{1-2\pi i \Theta X} \Big)^{1/d} \\
 + \zeta(0,\boldsymbol{\alpha}) \log \Big( \frac{X}{a_d (1-2 \pi i \Theta X)} \Big)+  W_1(\Theta)+O_{f,R,\varepsilon} \big( X^{-R+\varepsilon} \big) \quad \text{if} \quad a_0 \neq 0
\end{multline}
where
\begin{multline} \label{W1}
W_1(\Theta)=\frac{a_d^{-1/d}}{d} x^{-1/d} \Gamma \Big( \frac{1}{d} \Big)  \sum_{k=1}^{\infty} \frac{(-1)^k \zeta(1/d+1-k)}{k!} a_0^{k}  x^{k} \\
 +\zeta^{\prime}(0,\boldsymbol{\alpha}) +\sum_{m=1}^{\infty } \frac{(-1)^{m}}{m}  \big( \frac{a_0}{a_d} \big)^m \zeta(m,\boldsymbol{\alpha}) \\
+\sum_{k=0}^{\infty}  \sum_{\substack{m=1 \\ m \not \equiv 1 \modu{d}}}^{dR}  c_m \frac{(-1)^k \zeta(\frac{1-m}{d}+1-k)}{k!} \Gamma \Big(\frac{1-m}{d} \Big) a_0^k a_d^{-\frac{1-m}{d}} x^{k-\frac{1-m}{d}} \\
+\sum_{k=0}^{\infty}  \sum_{m=1}^{R} \frac{(-1)^{k+m} \zeta(1-m-k)}{m! k!} a_0^k a_d^{m} x^{k+m} \zeta(-m,\boldsymbol{\alpha}). 
\end{multline}
and $\rho=\exp \big( {-1/X} \big)$. Further,
\begin{multline}  \label{auxphi2}
\Phi_f \big( \rho e (\Theta) \big)=\frac{a_d^{-1/d}}{d} \Gamma \Big(\frac{1}{d} \Big) \zeta \Big( 1+\frac{1}{d} \Big) \Big(\frac{X}{1-2\pi i \Theta X} \Big)^{1/d} \\
 + \zeta(0,\boldsymbol{\alpha}) \log \Big( \frac{X}{a_d (1-2 \pi i \Theta X)} \Big)+  W_0(\Theta)+O_{f,R,\varepsilon} \big( X^{-R+\varepsilon} \big) \quad \text{if} \quad a_0=0, 
\end{multline}
where
\begin{multline} \label{W2}
W_0(\Theta)=\zeta^{\prime}(0,\boldsymbol{\alpha}) +\sum_{\substack{m=1 \\ m \not \equiv 1 \modu{d}}}^{dR}  c_m \zeta(\frac{1-m}{d}) \Gamma \Big(\frac{1-m}{d} \Big) a_d^{-\frac{1-m}{d}} x^{-\frac{1-m}{d}} \\
+\sum_{m=1}^{R} \frac{(-1)^{m} \zeta(1-m)}{m!} a_d^{m} x^{m} \zeta(-m,\boldsymbol{\alpha}). 
\end{multline}
\end{lemma}

\begin{remark}
In this proof, the implied constants occuring in estimates used to show uniform convergence of a sum of funtions or an integrand are allowed to depend on a fixed $X$. Bounds for errors coming from integrals used to obtain the error terms in \eqref{auxphi} and \eqref{auxphi2} are uniform in $X$. Note that \eqref{pred2} and \eqref{stirl} are repeatedly used in the proof of Lemma \ref{1}.
\end{remark}

\begin{proof}
First consider the case $a_0>0$. Recall
\begin{equation*}
\Phi_f(\rho e(\Theta) )= \sum_{j=1}^{\infty} \sum_{n=1}^{\infty} \frac{1}{j} \exp \big({- j f(n) x }\big).
\end{equation*}
This can be re-written as
\begin{equation*}
\Phi_f (\rho e(\Theta))= \sum_{j=1}^{\infty} \frac{\exp ({- j a_0 x} )}{j}  \sum_{n=1}^{\infty} \exp \big({-j(f(n)-a_0) x} \big).
\end{equation*}
Using the Cahen--Mellin transform \cite[Entry 3.1]{ober} the above becomes
\begin{align}
\Phi_f \big(\rho e(\Theta) \big)&=  \frac{1}{2 \pi i} \sum_{j=1}^{\infty} \frac{\exp ({- j a_0 x} )}{j} \sum_{n=1}^{\infty} \int_{(c)}  \Gamma(s) \big(j (f(n)-a_0) x \big)^{-s} ds \nonumber \\
&=\frac{1}{2 \pi i} \sum_{j=1}^{\infty} \exp (- j a_0 x )   \sum_{n=1}^{\infty} \int_{(c)} j^{-s-1}  \Gamma(s) x^{-s} a_d^{-s} n^{-s} \prod_{l=1}^{d-1} (n+\alpha_l)^{-s} ds \label{swapint},
\end{align}
for any $c>0$. Here we choose $c:=1/d+\varepsilon$ for any $\varepsilon>0$. By an argument in \cite{V2} we know that
\begin{equation} \label{pred2}
\Big | \Big(\frac{X}{1-2 \pi i \Theta X} \Big)^{s} \Big | \leq (X \Delta )^{\sigma}  \exp \Big( { |t| \Big( {\frac{\pi }{2} - \Delta } \Big)} \Big).
\end{equation}
We also have Stirling's bound
\begin{equation} \label{stirl}
|\Gamma(s) | \ll |s|^{\sigma-\frac{1}{2}} \exp \Big( {-\frac{\pi}{2} |t| } \Big) \quad \text{where} \quad s=\sigma+it.
\end{equation}
For a fixed $X$, each integrand in \eqref{swapint} is bounded above by
\begin{equation*}
\ll_f |c+it|^{c-\frac{1}{2}} j^{-c-1}  \exp \big({-\Delta |t|} \big) n^{-c} \prod_{l=1}^{d-1} |n+\alpha_l|^{-c}  \ll_{f,X} n^{-c} \prod_{l=1}^{d-1} |n+\alpha_l|^{-c} .
\end{equation*}
Uniform convergence ensures we may interchange the summation over $n$ and integration in \eqref{swapint} to obtain
\begin{equation*}
\Phi_f \big(\rho e(\Theta) \big)=\frac{1}{2 \pi i} \sum_{j=1}^{\infty} \int_{(c)} \frac{\exp \big({-j a_0 x} \big)}{j^{s+1}} \Gamma(s) x^{-s} a_d^{-s} \zeta(s,\boldsymbol{\alpha}) ds.
\end{equation*}
By a similar uniform convergence argument using Lemma \ref{MWbound} we may interchange the summation over $j$ and integration to obtain 
\begin{equation} \label{keyint}
\Phi_f \big(\rho e(\Theta) \big)=\frac{1}{2 \pi i} \int_{(c)}  \Li_{s+1}(e^{-a_0 x}) \Gamma(s) x^{-s} a_d^{-s} \zeta(s,\boldsymbol{\alpha}) ds,
\end{equation}
where $\operatorname{Li}_s(z)$ is the polylogarithm function \cite{R}
\begin{equation*}
\operatorname{Li}_s(z):=\sum_{k=1}^{\infty} \frac{z^k}{k^s} \quad \text{for} \quad s \in \mathbb{C} \quad \text{and} \quad |z|<1.
\end{equation*}

We now employ the polylogarithm identity \cite[pg. 1050]{R}
\begin{equation} \label{polylog}
\Li_s(e^{\mu})=\Gamma(1-s) (-\mu)^{s-1}+\sum_{k=0}^{\infty} \frac{\zeta(s-k)}{k!} \mu^k,
\end{equation}
valid for all $|\mu|<2 \pi$ and $s \neq 1,2,3,\ldots$. Since $|\Theta| \leq X^{\frac{L}{d}-1}$, we have  $|a_0 x|<2 \pi$ for all $X$ sufficiently large. Also $s+1$ is never a positive integer on the line of integration $(c)$, so using \eqref{polylog} in \eqref{keyint} we have 
\begin{equation} \label{key}
\Phi_f \big(\rho e(\Theta) \big)=\frac{1}{2 \pi i} ( I_1+I_2)
\end{equation}
where
\begin{align}
I_1&:=\int_{(c)}  \Gamma(s) \Gamma(-s) \big( \frac{a_0}{a_d} \big)^s \zeta(s,\boldsymbol{\alpha}) ds \nonumber \\
I_2&:=\int_{(c)} \sum_{k=0}^{\infty} \frac{(-1)^k \zeta(s+1-k)}{k!} \Gamma(s) a_0^k a_d^{-s} x^{k-s} \zeta(s,\boldsymbol{\alpha}) ds \label{I2}.
\end{align}
\subsubsection{Integral $I_1$} Here is where we use the hypothesis that $|a_0/a_d|<1$.  For $P>1>c$ fixed, we extend the contour to the rectangle $\mathcal{H}_{P,T}$ with vertices $[P-iT,P+iT,c-iT,c+iT]$, oriented clockwise.  For all $s \in \mathcal{H}_{P,T}$ note that $\zeta(s,\boldsymbol{\alpha})= O(1)$ uniformly in $P$. Using \eqref{stirl} we see that the integral on the horizontal sides of $\mathcal{H}_{P,T}$ are bounded above by 
\begin{equation*}
\ll P |a_0/a_d|^{c}  |c +iT   |^{-1} \exp ({-\pi |T|})  \rightarrow 0 \quad \text{as} \quad T \rightarrow \infty.
\end{equation*}
On the vertical line $\Re(s)=P$ the integral is bounded above by 
\begin{equation*}
 |a_0/a_d|^P \int_{0}^{\infty} |P+it|^{-1} \exp ({-\pi |t|})  dt \ll_f P^{-1} |a_0/a_d|^P.
\end{equation*}
Inside of $\mathcal{H}_{P,T}$ the integrand has poles at $s=1,2,\ldots, \lfloor P \rfloor$. By Cauchy's residue theorem (keeping in mind the orientation)
\begin{equation*}
I_1=\sum_{m=1}^{\lfloor P \rfloor } \frac{(-1)^{m}}{m}  \big( \frac{a_0}{a_d} \big)^m \zeta(m,\boldsymbol{\alpha})+O_f \big(P^{-1} |a_0/a_d|^P \big).
\end{equation*}
Now allowing $P \rightarrow \infty$ we obtain
\begin{equation*}
I_1=\sum_{m=1}^{\infty } \frac{(-1)^{m}}{m}  \big( \frac{a_0}{a_d} \big)^m \zeta(m,\boldsymbol{\alpha}).
\end{equation*}

\subsubsection{Integral $I_2$}  In order to interchange the integration and the summation over $k$ in \eqref{I2} we need to show uniform convergence of the sum on the domain of integration. By the asymmetrical version of functional equation for $\zeta$, \eqref{stirl} and the fact that $\zeta(k-s)=O(1)$ uniformly with respect to $k$ for all $k \geq 2$ and $s \in (c)$. Thus the following holds for $k \geq 2$ and $s \in (c)$:
\begin{align} \label{examp}
\frac{| \zeta(s+1-k) | |a_0 x|^k }{k!} &= \frac{|a_0 x|^k}{k!} \Big | 2^{s+1-k} \pi^{s-k} \sin \Big(\frac{\pi(s+1-k)}{2} \Big) \Gamma(k-s) \zeta(k-s)  \Big | \nonumber \\
& \ll_f \frac{|a_0 x|^k}{k!} (2 \pi)^{-k} \exp \Big( \frac{\pi}{2} |t| \Big) |k-s|^{k-c-1/2} \exp \Big({-\frac{\pi}{2} |t|} \Big) \nonumber \\
& \ll _f\frac{|a_0 x|^k}{k!} (2 \pi)^{-k} |k-s|^{k} \nonumber \\
& \ll_f \Big( \frac{e |a_0 x|}{2\pi} \Big)^k  2^k \Big(1+\Big(\frac{|t|}{k} \Big)^k \Big) \\
& \ll_f \Big( \frac{e |a_0 x|}{\pi} \Big)^k  \Big(1+\Big(\frac{|t|}{k} \Big)^k \Big).
\end{align} 
The last lines follow by the triangle inequality and the asymptotic
\begin{equation*}
k! \sim \sqrt{2 \pi k} \Big( \frac{k}{e} \Big)^k \quad \text{as} \quad k \rightarrow \infty.
\end{equation*}
Applying \eqref{pred2}, Lemma \ref{MWbound} and \eqref{stirl}, the following holds for $|t| \geq e$
\begin{equation*} \label{ub}
|\Gamma(s) x^{-s} a_d^{-s} \zeta(s,\boldsymbol{\alpha}) |  \ll_f (X \Delta)^{c} |c+it|^{c-\frac{1}{2}} |t|^{O_{\boldsymbol{\alpha}}(1)} \exp \big( {-\Delta |t|} \big). 
\end{equation*}
Thus for $|t|>e$ each summand inside the integrand of $I_2$ is bounded by 
\begin{equation} \label{habd}
 \ll_{f,X}  \Big( \frac{e |a_0 x|}{\pi} \Big)^k |t|^{O_{\boldsymbol{\alpha}}(1)} \exp \big({-\Delta |t|} \big) \bigg(1+\frac{|t|^{k}}{k^k}  \bigg). 
\end{equation}
Set $\tilde{t}:=X^{-\frac{L}{d}} t$. Then the right side \eqref{habd} is bounded above
\begin{align}
& \ll_{f,X} \Big( \frac{e |a_0 x| X^{\frac{L}{d}} }{\pi} \Big)^k |t|^{O_{\boldsymbol{\alpha}}(1)} \exp \big({-\Delta X^{\frac{L}{d}} |\tilde{t}|} \big) \bigg(1+\frac{|\tilde{t}|^{k}}{k^k}  \bigg) \nonumber \\
& \ll_{f,X} \Big( \frac{e |a_0 x| X^{\frac{L}{d}} }{\pi} \Big)^k  \exp \Bigg({-\Delta X^{\frac{L}{d}} |\tilde{t}|+ \log \Big(1+\frac{|\tilde{t}|^{k}}{k^k} \Big) +O_{\boldsymbol{\alpha}} \big(\log |\tilde{t}| \big)} \Bigg) \label{auxbf}.
\end{align}
Since $|\Theta| \leq X^{\frac{L}{d}-1}$ we see that $\Delta X^{\frac{L}{d}}$ is bounded below by some absolute positive constant. Thus for all $|\tilde{t}|>1$,
\begin{equation} \label{analo}
-\Delta X^{\frac{L}{d}} |\tilde{t}|+\log \Big(1+\Big( \frac{|\tilde{t}|}{k} \Big)^k \Big)+O_{\boldsymbol{\alpha}} \big (\log |\tilde{t}| \big) \leq C_{\boldsymbol{\alpha}} k,
\end{equation}
for some absolute constant $C_{\boldsymbol{\alpha}}>0$. Thus \eqref{auxbf} is bounded above by 
\begin{equation} \label{inifI2}
\ll_{f,X} \Big( \frac{e |a_0 x| X^{\frac{L}{d}} C_{\boldsymbol{\alpha}} }{\pi} \Big)^k.
\end{equation}

For all $s=c+it$, with $|t| \leq X^{\frac{L}{d}}$, we have $|\zeta(s,\boldsymbol{\alpha})|=O_{\boldsymbol{\alpha}}(1)$ uniformly in $t$ by the Extreme Value Theorem and the fact that $\zeta(c+it,\alpha)$ is absolutely convergent. So the left hand side of \eqref{ub} is bounded above by 
\begin{equation*}
\ll_{f,X} |c+it|^{c-\frac{1}{2}} \exp \big({-\Delta |t|} \big) \ll_{f,X} 1.
\end{equation*}
Thus each summand inside the integrand of $I_2$ for $|t| \leq X^{\frac{L}{d}}$ is bounded by
\begin{align} 
& \ll_{f,X} \Big( \frac{e |a_0 x|}{ \pi} \Big)^k \exp \Bigg(  \log \Big(1+\Big(\frac{X^{\frac{L}{d}}}{k} \Big)^k \Big) \Bigg) \nonumber \\
& \ll_{f,X} \Big( \frac{e |a_0 x| X^{\frac{L}{d}}}{\pi} \Big)^k \label{unifI22}.
\end{align}

Since $|\Theta| \leq X^{\frac{L}{d}-1}$, 
\begin{equation*}
|x|X^{\frac{L}{d}} \leq \Big(\frac{1}{X}+2 \pi X^{\frac{L}{d}-1} \Big) X^{\frac{L}{d}} \leq 2 \pi \Big(X^{\frac{L}{d}-1}+X^{\frac{2L}{d}-1}  \Big) 
\end{equation*}
for all $d \geq 2$. Thus $|x|X^{\frac{L}{d}}$ is small for all sufficiently large fixed $X$.  Using this observation in \eqref{inifI2} and \eqref{unifI22} we conclude that each summand in \eqref{I2} is bounded above by
\begin{equation*}
\ll_{f,X} \Big(\frac{1}{2} \Big)^k
\end{equation*}
for all sufficiently large fixed $X$. Thus the sum over $k$ in \eqref{I2} converges uniformly on (c) and we can interchange the integration and summation in $I_2$. 

After performing the interchange of integration and summation in $I_2$, we extend the contour in each summand to the rectangle $\mathcal{M}_{T,R}$ with vertices $[-R-iT, -R+iT, c+iT,c-iT]$ for $R>0$ fixed, oriented anti-clockwise. We arrive at
\begin{equation}  \label{swap}
I_3:= \sum_{k=0}^{\infty} \int_{\mathcal{M}_{T,R}} \frac{(-1)^k \zeta(s+1-k)}{k!} \Gamma(s) a_0^k a_d^{-s} x^{k-s} \zeta(s,\boldsymbol{\alpha}) ds.
\end{equation}

All summands have a simple pole at $s=1/d$ inside $\mathcal{M}_{T,R}$, with total residue 
\begin{equation} \label{res1}
\frac{a_d^{-1/d}}{d} x^{-1/d} \Gamma \Big( \frac{1}{d} \Big)  \sum_{k=0}^{\infty} \frac{(-1)^k \zeta(1/d+1-k)}{k!} a_0^{k}  x^{k},
\end{equation}
since the residue of $\zeta(s,\boldsymbol{\alpha})$ at $s=1/d$ is $1/d$.

The $k=0$ summand in \eqref{swap} has a double pole at $s=0$ with residue
\begin{equation} \label{res2}
\zeta(0,\boldsymbol{\alpha}) \log \Big( \frac{X}{a_d(1-2 \pi i \Theta X)} \Big)+\zeta^{\prime}(0,\boldsymbol{\alpha}).
\end{equation}

All summands for $k \geq 1$ in \eqref{swap} have a simple pole at $s=0$ with total residue
\begin{equation} \label{res3}
\zeta(0,\boldsymbol{\alpha}) \sum_{k=1}^{\infty} \frac{(-1)^k \zeta(1-k)}{k!} a_0^k x^k.  
\end{equation}

Each summand has a potential simple pole at $s=(1-m)/d$ for $m \in \mathbb{N}$ with $m \not \equiv 1 \modu{d}$, with total residue 
\begin{equation} \label{res4}
\sum_{k=0}^{\infty}  \sum_{\substack{m=1 \\ m \not \equiv 1 \modu{d}}}^{dR}  c_m \frac{(-1)^k \zeta(\frac{1-m}{d}+1-k)}{k!} \Gamma \Big(\frac{1-m}{d} \Big) a_0^k a_d^{-\frac{1-m}{d}} x^{k-\frac{1-m}{d}}
\end{equation}
where $c_m$ is the residue of $\zeta(s,\boldsymbol{\alpha})$ at $s=(1-m)/d$. The method for finding these residues is explained in \cite{MW}.

Each summand has a simple pole $s=-m$ for $m \in \mathbb{N}$, with total residue
\begin{equation}  \label{res5}
\sum_{k=0}^{\infty}  \sum_{m=1}^{R} \frac{(-1)^{k+m} \zeta(1-m-k)}{m! k!} a_0^k a_d^{m} x^{k+m} \zeta(-m,\boldsymbol{\alpha}).
\end{equation}

For all $T>0$, $I_3$ is the sum of \eqref{res1}, \eqref{res2}, \eqref{res3}, \eqref{res4} and \eqref{res5} by Cauchy's residue theorem. 

We now reconcile with $I_2$ with $I_3$.  We write \eqref{swap} as 
\begin{equation} \label{interchangeI3}
I_3= \sum_{k=0}^{\infty} \Big( \int_{-R+iT}^{-R-iT} +\int_{(c)}+\int_{L_{R,T}}+\int_{L_{R,-T}} \Big)  \frac{(-1)^k \zeta(s+1-k)}{k!} \Gamma(s) a_0^k a_d^{-s} x^{k-s} \zeta(s,\boldsymbol{\alpha}) ds,
\end{equation}
where $L_{R,T}$ and $L_{R,-T}$ denote horizontal sides of $\mathcal{M}_{R,T}$, oriented appropriately. We now apply Lemma $\ref{MWbound}$ and a similar argument used to establish \eqref{examp}. From this we see that 
\begin{multline} \label{domco}
\Big | \int_{L_{R,\pm T}} \Big |  \ll_{f,R,X}  \sum_{k=1}^{\infty} \Big( \frac{e |a_0 x|}{\pi} \Big)^k (k+R)^R \exp  \Bigg(-\Delta |T| + \log \Big(1+ \Big( \frac{|T|}{k} \Big)^{k+R}  \Big)+O_{\boldsymbol{\alpha},R}(\log |T|) \Bigg).
\end{multline} 
Without loss of generality we may assume that $|T| \geq X^{\frac{L}{d}}$. Arguing as above we see that the summand of \eqref{domco} is bounded above by \eqref{inifI2} with $C_{\boldsymbol{\alpha}}$ replaced with $C_{\boldsymbol{\alpha},R}$. By the Dominated Convergence Theorem we can let $T \rightarrow \infty$ in each of the summands and the hence the limit of \eqref{domco} as $T \rightarrow \infty$ is $0$.

All that remains to do is bound the contribution to \eqref{interchangeI3} from the line $\Re(s)=-R$. Using \eqref{pred2}, \eqref{stirl}, Lemma \ref{MWbound} and performing a similar computation used to establish \eqref{examp},  the integrand of \eqref{interchangeI3} (say for $|t| \geq 1)$ is bounded above by 
\begin{align} 
&\ll_{f,R} (X \Delta)^{-R} |-R+it|^{-R-\frac{1}{2}} |t|^{O_{\boldsymbol{\alpha},R}(1)} \exp \big({- \Delta |t|} \big) \frac{|\zeta(1-R-k+it)|}{k!} |a_0 x|^k \nonumber \\
& \ll_{f,R}  (X \Delta)^{-R}  |t|^{O_{\boldsymbol{\alpha},R}(1)}  \Big( \frac{e |a_0 x|}{\pi} \Big)^k (k+R)^R \Big(1+\frac{|t|^{k+R}}{k^{k+R}} \Big) \exp \big({- \Delta |t|} \big) \label{unif2}. 
\end{align}
When $|t| \leq 1$ and $\text{Re}(s)=-R$ (with $R \notin \frac{1}{d} \mathbb{Z}$), $|\zeta(s,\boldsymbol{\alpha})|=O_{\boldsymbol{\alpha},R}(1)$ by the Extreme Value Theorem. Thus the integrand of \eqref{interchangeI3} in this case is bounded above by
\begin{align} 
& \ll_{f,R} (X \Delta)^{-R} \Big( \frac{e |a_0 x|}{\pi} \Big)^k (k+R)^{R} |-R+it|^{-R-\frac{1}{2}} \exp({-\Delta |t|}) \nonumber \\
& \ll_{f,R} (X \Delta)^{-R} \Big( \frac{e |a_0 x|}{\pi} \Big)^k (k+R)^{R}.  \label{unif1}
\end{align}
Using \eqref{unif2} and \eqref{unif1} we can use a similar uniform convergence argument appearing above to interchange the summation over $k$ and the integral $\int_{-R+iT}^{-R-iT}$ in \eqref{interchangeI3}. Doing this we obtain
\begin{multline} \label{final2}
\Big | \int_{-R+iT}^{-R-iT} \sum_{k} \Big | \ll_{f,R} (X \Delta)^{-R} \Bigg( \int_{1}^{\infty} |t|^{O_{\boldsymbol{\alpha},R}(1)} \exp \big({- \Delta |t|} \big)  \\
\times \sum_k  \Big( \frac{e |a_0 x|}{\pi} \Big)^k (k+R)^R \Big(1+\frac{|t|^{k+R}}{k^{k+R}} \Big) dt +\sum_k \Big( \frac{e |a_0 x|}{\pi} \Big)^k(k+R)^{R} \Bigg). 
\end{multline}
Making a change of variable, the right hand side of \eqref{final2} is bounded above by 
 \begin{align*}
& \ll_{f,R} (X \Delta)^{-R} \sum_{k=0}^{\infty} (k+R)^R \Big(  \frac{e |a_0 x|}{ \pi} \Big)^k \Bigg(\frac{\Gamma \big(k+O_{\alpha,R}(1) \big)}{ \big |\Delta \big|^{k+O_{\boldsymbol{\alpha},R}(1)} k^k}+\Gamma \big( O_{\boldsymbol{\alpha}, R}(1) \big) \Bigg) \\
& \ll_{f,R} X^{-R} \Delta^{-R-O_{\boldsymbol{\alpha},R}(1)} \sum_{k=0}^{\infty} (k+R)^R \Big(  3 e |a_0 x| X^{\frac{L}{d}} \Big)^k \\
& \ll_{f,R} X^{-R+O_{\boldsymbol{\alpha},R}(L/d)}\\
& \ll_{f,R} X^{-R+\varepsilon},  
\end{align*}
where the last inequalities holds for all sufficiently large fixed $X$ and $L$ chosen small enough depending on $R, \boldsymbol{\alpha}$ and $\varepsilon$. 

In the case $a_0=0$, we need only analyse the integral
\begin{equation*}
I_3:=\int_{(c)} \Gamma(s) \zeta(s+1) \zeta(s,\boldsymbol{\alpha}) a_d^{-s} x^{-s} ds
\end{equation*}
which can be handled by similar arguments used above. Putting all of the above together we obtain the result.
\end{proof}

\subsection{Major arc estimates}
Let $f \in \mathbb{Z}[y]$ as is in the Theorem \ref{theorem1}. Since all the coefficients of $f$ are non-negative, there is a smooth $\psi:\mathbb{R}_{\geq 0} \rightarrow \R_{\geq 0}$ such that $(\psi \circ f)(x)=x$. In our implementation of the circle method in Section \ref{maintheorems}, for large $U>0$ we will need major arc estimates for the exponential sum
\begin{equation*}
\mathcal{F}(\Theta):=\sum_{y=1}^{\psi(U)} e \big(\Theta f(y) \big).
\end{equation*}
This is more general than the linear combination of single monomial and linear term originally considered for Waring's problem, see \cite[Theorem~4.1]{V1}.  To this end, fix $U>0$ large and for $\Theta \in \R, a,b \in \mathbb{Z}$ and $q \in \N$, we define the auxiliary exponential sums
\begin{align*}
\mathcal{S}(q,a,b,f)&:=\sum_{y=1}^{q} e \Big( \frac{a f(y)+by}{q} \Big), \\
\mathcal{S}(q,a,f)&:=\mathcal{S}(q,a,0,f).
 \end{align*}
Furthermore, we define the integral function
\begin{align*}
v_{f}(\beta)&=\int_{0}^{\psi(U)} e \big(\beta f(\gamma) \big) d \gamma,
\end{align*}
and
\begin{equation*}
\mathcal{V}(\Theta,q,a,f):=q^{-1} \mathcal{S}(q,a,f) v_f(\Theta-a/q), \\
\end{equation*}

The proof of Lemma \ref{sing} will be a modification of \cite[Theorem~4.1]{V1}. Where Vaughan appeals to Hua's bound in \cite[Theorem~4.1]{V1}, we will instead appeal to Weyl's inequality for exponential sums whose argument is a polynomial. The property that $f \in \mathbb{Z}[y]$ has non-negative coefficients and no linear term will play a crucial role in the proof of Lemma \ref{sing}. 

\begin{lemma} \label{sing}
Let $f(y) \in \mathbb{Z}[y]$ be as in Theorem \ref{theorem1}. Suppose $a,q \in \N$ such that $(a,q)=1$ and $\Theta=a/q+\beta$ where $|\beta | \leq 1/q$. Then for any $\varepsilon>0$ we have
\begin{equation} \label{ineq1}
\mathcal{F}(\Theta)-\mathcal{V}(\Theta,q,a,f) \ll_{f,\varepsilon} q^{1-2^{1-d}+\varepsilon}(1+U |\beta|)^{\frac{1}{2}}.
\end{equation}
\end{lemma}
\begin{proof}
For brevity put $M:=\psi(U)$. We write
\begin{align}
\mathcal{F}(\Theta)&=\sum_{y \leq M} e \big(\beta f(y) \big) \sum_{\substack{m=1 \\ m \equiv y \modu{q} }}^{q} e \Big( \frac{a f(m)}{q} \Big) \nonumber \\
&=q^{-1} \sum_{-q/2<b \leq q/2} \mathcal{S}(q,a,b,f) F(b)  \label{first},
\end{align}
where 
\begin{equation*}
F(b):=\sum_{y \leq M} e \big( \beta f(y)-b y/q \big).
\end{equation*}
For each $-q/2<b \leq q/2$, the function $\beta f^{\prime}(\gamma)-b/q$ is monotonic in the variable $\gamma \in [0,M]$. Thus \cite[Lemma~4.2]{V1} of Vaughan can be applied to the above interval for $\gamma$ with $H_1=-H_2=-H$ where $H = \lceil |\beta| f^{\prime}(M)+1 \rceil$. This yields
\begin{equation*}
F(b)=\sum_{h=-H}^H I(b+hq)+O_f \big( \log(2+H)   \big)
\end{equation*}
where 
\begin{equation*}
I(c):=\int_{0}^M e \big( \beta f(\gamma)-c \gamma/q \big)  d \gamma.
\end{equation*}
Since the polynomial is fixed and $(a,q)=1$, we have $(a a_d, q)=O_f(1)$ uniformly in $q$.  Thus for each $q$ there exists $\tilde{a}, \tilde{q} \in \mathbb{Z}$ such that $\tilde{q}>0$, $(\tilde{a},\tilde{q})=1$, $\tilde{q} \gg_f q$ and  $aa_d/q=\tilde{a}/\tilde{q}$. This rational number is the leading coefficient of the polynomial $(af(y)+by)/q$, so we can apply Weyl's inequality \cite[Theorem~4.3]{NA} to obtain
\begin{equation} \label{wbound}
|\mathcal{S}(q,a,b,f)| \ll_{f,\varepsilon} q^{1+\varepsilon} \tilde{q}^{-2^{1-d}} \ll_{f,\varepsilon} q^{1-2^{1-d}+\varepsilon}
\end{equation}
for any fixed $\varepsilon>0$. Note that \eqref{wbound} is uniform in $b$ because the implied constant in \cite[Theorem~4.3]{NA} depends only on $\varepsilon$ and the degree $d$ of the polynomial. Thus
\begin{multline} \label{mainterm}
\mathcal{F}(\Theta)-q^{-1} \mathcal{S}(q,a,f) v_f(\beta) \\
=q^{-1} \sum_{\substack{-B<b \leq B \\ b \neq 0}} \mathcal{S}(q,a,b,f) I(b)+O_{f,\varepsilon} \big( q^{1-2^{1-d}+\varepsilon} \log(2+H) \big)
\end{multline}
where $B=\big(H+\frac{1}{2} \big)q$. Now we consider \eqref{ineq1}. Note that the error term in \eqref{mainterm} is acceptable. 

Using integration by parts we see that the contribution to $I(b)$ from those $\gamma \in [0,M]$ such that 
\begin{equation} \label{goodgamma}
\big | \beta f^{\prime}(\gamma) -b/q \big | \geq  \frac{1}{2} \Big |\frac{b}{q} \Big |
\end{equation}
is $\ll q/|b|$. Thus the total contribution from $\gamma$ to \eqref{mainterm} satisfying \eqref{goodgamma} is bounded by
\begin{equation} \label{fine2}
\ll_{f,\varepsilon} q^{1-2^{1-d}+\varepsilon} \log(2B) \ll_{f,\varepsilon} q^{1-2^{1-d}+\varepsilon} \big(1+U |\beta | \big)^{\frac{1}{2}}.
\end{equation}

For $\gamma$ not satisfying \eqref{goodgamma} we must have
\begin{equation} \label{remgam}
| \beta f^{\prime}(\gamma)-b/q | \leq \frac{ |b|}{2 q}.
\end{equation} 
Such $\gamma$ must satisfy
\begin{equation} \label{divide}
\frac{|b|}{2 q} \leq  | \beta | f^{\prime}(\gamma)  \leq \frac{ 3|b|}{2q}.
\end{equation}
Thus for sufficiently large $M>1$ we have
\begin{equation*}
|b| \ll_f q |\beta| f^{\prime}(M) \ll_f q |\beta| M^{d-1}.
\end{equation*}
For such a $b$ let 
\begin{equation*}
\delta:=|\beta|^{\frac{1}{2d-2}} \Big( \frac{|b|}{q} \Big)^{\frac{d-2}{2d-2}}.
\end{equation*}
Again applying integration by parts, the contribution to $I(b)$ from the $\gamma$ satisfying \eqref{remgam} with $|\beta f^{\prime}(\gamma)-b/q| \geq \delta$ is $\ll \delta^{-1}$. We now treat the remaining $\gamma$ satisfying \eqref{remgam}. In other words, such $\gamma$ satisfy $|\beta f^{\prime}(\gamma)-b/q| \leq \delta$. We apply the triangle inequality and see such $\gamma$ must lie in an interval $[\gamma_1,\gamma_2]$ satisfying: 
\begin{align}
2 \delta &\geq |\beta|  | f^{\prime}(\gamma_2)-f^{\prime}(\gamma_1)  | \nonumber \\
& \geq  |\beta| \sum_{j=1}^{d} a_j j \big(\gamma_2^{j-1} -\gamma_1^{j-1}\big) \nonumber \\
& \geq  |\beta|(\gamma_2-\gamma_1) \sum_{j=2}^{d} a_j j \gamma_2^{j-2} \nonumber \\
& \gg_f  |\beta|(\gamma_2-\gamma_1) \frac{f^{\prime}(\gamma_2)}{\gamma_2}  \label{keya1},
\end{align}
where we have used the facts that $a_j \geq 0$ for $0 \leq j \leq d$, $\gamma_2>\gamma_1$ and $a_1=0$. 

Now consider the case $\gamma_2 \geq 1$. We have 
\begin{equation} \label{keya4} 
f^{\prime}(\gamma_2) \asymp_f \gamma_2^{d-1} \quad \text{and} \quad \frac{f^{\prime}(\gamma_2)}{\gamma_2} \asymp_f \gamma_2^{d-2}.
\end{equation}
Thus 
\begin{equation} \label{keya2}
|\beta|(\gamma_2-\gamma_1) \frac{f^{\prime}(\gamma_2)}{\gamma_2} \gg_{f} |\beta| (\gamma_2-\gamma_1)  \gamma_2^{d-2}. 
\end{equation}
Furthermore \eqref{divide} and \eqref{keya4} imply that 
\begin{equation} \label{keya3}
\gamma_2^{d-2} \asymp_f  \Big (\frac{|b|}{q |\beta|} \Big)^{\frac{d-2}{d-1}}.
\end{equation}
Combining \eqref{keya1}, \eqref{keya2} and \eqref{keya3} we have 
\begin{equation*}
|\gamma_2-\gamma_1| \ll_f \delta^{-1}.
\end{equation*}

Consider the case $0<\gamma_2<1$. Then
\begin{equation} \label{keya5}
|\beta|(\gamma_2-\gamma_1) \frac{f^{\prime}(\gamma_2)}{\gamma_2}  \gg_f |\beta|(\gamma_2-\gamma_1) f^{\prime}(\gamma_2) \gg_f (\gamma_2-\gamma_1) \frac{|b|}{q},
\end{equation}
where the last inequality follows from \eqref{divide}. Then \eqref{keya1} and \eqref{keya5} imply that 
\begin{equation*}
\gamma_2-\gamma_1 \ll_f \delta^{-1}  |b|^{-\frac{2}{2d-2}}  |\beta q|^{\frac{2}{2d-2}} \ll_f \delta^{-1},
\end{equation*}
where the last line follows from the facts that $|\beta| \leq 1/q$ and $b \neq 0$ is an integer.

Thus the total contribution to \eqref{mainterm} from $\gamma$ satisfying \eqref{remgam} is 
\begin{align}
 & \ll_{f,\varepsilon} \sum_{0<b \ll_f q |\beta| M^{d-1}} q^{-2^{1-d}+\varepsilon} | \beta |^{-\frac{1}{2d-2}} \Big( \frac{q}{|b|} \Big)^{\frac{d-2}{2d-2}} \nonumber \\
 & \ll_{f,\varepsilon}  q^{\frac{d-2}{2d-2} -2^{1-d}+\varepsilon} | \beta |^{-\frac{1}{2d-2}} \Big(q |\beta| M^{d-1} \Big)^{1-\frac{d-2}{2d-2}} \nonumber \\
 & \ll_{f,\varepsilon} q^{1-2^{d-1}+\varepsilon} |\beta|^{\frac{1}{2}} M^{\frac{d}{2}} \nonumber \\
 & \ll_{f,\varepsilon} q^{1-2^{d-1}+\varepsilon} |\beta|^{\frac{1}{2}} U^{\frac{1}{2}} \label{fine}. 
\end{align}
Combining \eqref{mainterm}, \eqref{fine2} and \eqref{fine} yields the result.
\end{proof}

As explained in Section \ref{maintheorems}, most major arcs will actually be subsumed by the error coming from the minor arcs. The following two lemmas will ensure that we can conclude this.
\begin{lemma} \label{lemma3}
Let $f \in \mathbb{Z}[y]$ be as in Theorem \ref{theorem1}. Let $\Theta \in \R, a,q \in \N$ be such that $(a,q)=1$ and $\beta = \Theta - a/q$ and $|\beta| \leq 1/q$. Then for all sufficiently large $X$
\begin{equation*}
\big | \Phi _f \big(\rho e(\Theta ) \big) \big | \leq \Gamma \big( 1 + \frac{1}{d} \big) \Big( \frac{X}{a_d} \Big)^{1/d} \sum\limits_{j = 1}^\infty  \frac{|S(q_j,\tilde{a}_j,f)|}{j^{(d + 1)/d}q_j} +O_{f,\varepsilon} \Big(q^{1-2^{1-d} + \varepsilon} \log X { \big(1 + X| \beta | \big)^{\frac{1}{2}}} \Big)
\end{equation*}
where $q_j = q/(q,j)$, $\tilde{a}_j = a j/(q,j)$ and $\rho=\exp \big({-1/X} \big)$.
\end{lemma}
\begin{proof}
We first write $\Phi_f$ as
\begin{equation*}
 \Phi_f(\rho e(\Theta)) = \sum_{j=1}^\infty \sum_{n=1}^\infty \frac{1}{j} \exp(-j f(n)/X) \e ( j f(n) \Theta ).
\end{equation*}
For the first exponential we obtain
\begin{equation*}
\exp(-j f(n)/X) = \int_n^\infty f'(u)jX^{-1} \exp(-j f(u)/X) du,
\end{equation*}
and therefore
\begin{equation} \label{tail}
\Phi_f(\rho e(\Theta)) = \sum_{j=1}^\infty \frac{1}{j} \int_0^\infty f'(u)jX^{-1} \exp(-j f(u)/X) \sum_{n \le u}  \e ( j f(n) \Theta ) du.
\end{equation}
Recall that the hypothesis on $f$ of Theorem \ref{theorem1} implies that $a_j \geq 0$ for all $0 \leq j \leq d$. Trivially
\begin{align*}
& \Bigg | \int_0^\infty f'(u)jX^{-1} \exp({-j f(u)/X}) \sum_{n \le u}  \e \big( j f(n) \Theta \big) du \Bigg | \nonumber \\
 & \leq \int_0^\infty u f'(u)jX^{-1} \exp \big({-j f(u)/X} \big) du \\
 &  =\int_{0}^{\infty} \exp \big({-j f(u)/X} \big) du  \\ 
&\leq \int_{0}^{\infty} \exp \big({-ja_d u^d /X} \big) du \\
& \ll_{f}  \Big(\frac{X}{j} \Big)^{1/d} \int_{0}^{\infty}  \exp(-y^d)  dy   \\
& \ll_{f}  \Big(\frac{X}{j} \Big)^{1/d},
\end{align*}
where the third and fourth lines follow from integration by parts and the non-negativity of the polynomial coefficients respectively.

Let $J$ be a parameter of our choice and consider the tail of the sum over $j$ in \eqref{tail}. We have
\begin{multline} \label{tailexp}
\Bigg | \sum_{j=J+1}^\infty \frac{1}{j} \int_0^\infty f'(u)jX^{-1} \exp(-j f(u)/X) \sum_{n \le u}  \e ( j f(n) \Theta ) du \Bigg | \\
\ll \sum_{j=J+1}^\infty \frac{1}{j} \bigg(\frac{X}{j}\bigg)^{1/d} \ll \bigg(\frac{X}{J}\bigg)^{1/d}.
\end{multline}
It remains to consider
\begin{equation} \label{truncate}
\sum_{j=1}^J \frac{1}{j} \int_0^\infty f'(u)jX^{-1} \exp(-j f(u)/X) \sum_{n \le u}  \e ( j f(n) \Theta ) du.
\end{equation}
By Lemma \ref{sing} we have
\begin{equation*}
\sum_{n \le u} \e (j f(n) \Theta) = q_j^{-1}S(q_j,\tilde{a}_j,f) \int_0^{u} \e (j \beta f(\gamma)) d \gamma \\
 + O_{f,\varepsilon} \bigg (q_j^{1-2^{1-d}+\varepsilon} \big(1+f(u) j |\beta| \big)^{\frac{1}{2}} \bigg).
\end{equation*}
Thus \eqref{truncate} becomes
\begin{equation} \label{mt}
\sum_{j=1}^J \frac{S(q_j,\tilde{a}_j,f)}{j q_j} \int_0^\infty f'(u) j X^{-1} \exp(-j f(u)/X) \int_0^{u} \e \big(j \beta f(\gamma) \big) d \gamma du + E
\end{equation}
where
\begin{align*}
|E| & \ll_{f,\varepsilon} \sum_{j=1}^J \frac{q_j^{1-2^{1-d}+\varepsilon}}{j} \Big( \int_0^1+\int_{1}^{\infty} \Big) f'(u) j X^{-1} \exp(-j f(u)/X) \Big(1+ \big(f(u) j |\beta| \big)^{\frac{1}{2}} \Big) du  \\
& \ll_{f,\varepsilon} q^{1-2^{1-d}+\varepsilon} \Big(\frac{J}{X}+\frac{J^{\frac{3}{2}} |\beta|^{\frac{1}{2}} }{X}+ \sum_{j=1}^{J} \frac{1}{j} \int_{1}^{\infty} f'(u) j X^{-1}  \exp({-j f(u)/X}) \Big(1+ \big(f(u) j |\beta| \big)^{\frac{1}{2}} \Big) du \Big) \\
& \ll_{f,\varepsilon} q^{1-2^{1-d}+\varepsilon}  \Big(\frac{J}{X}+\frac{J^{\frac{3}{2}} |\beta|^{\frac{1}{2}} }{X}+\sum_{j=1}^{J} \frac{1}{j} \int_{1}^{\infty} d u^{d-1} j X^{-1} \exp({-j a_d u^d/X}) \Big(1+ \big(u^d j |\beta|)^{\frac{1}{2}} \Big) du \Big) \\
& \ll_{f,\varepsilon} q^{1-2^{1-d}+\varepsilon} \Big(\frac{J}{X}+\frac{J^{\frac{3}{2}} |\beta|^{\frac{1}{2}} }{X}+\log J \big(1+|\beta|^{\frac{1}{2}} X^{\frac{1}{2}} \big)  \Big),  
\end{align*}
where the last line follows from \cite[p.~27]{G}. Choosing $J=X$ yields the bound
\begin{equation} \label{first}
|E| \ll_{f,\varepsilon} q^{1-2^{1-d}+\varepsilon} \log X \Big(1+|\beta|^{\frac{1}{2}} X^{\frac{1}{2}} \Big).
\end{equation}

We now turn our attention to the main terms in \eqref{mt}, integrating by parts we obtain  
\begin{equation*}
\sum_{j=1}^X \frac{S(q_j,\tilde{a}_j,f)}{j q_j} \int_{0}^{\infty} \exp \big( -j f(u) X^{-1}(1-2\pi i \beta X) \big) du.
\end{equation*}
This is bounded above by   
\begin{align}
& \sum_{j=1}^X \frac{|S(q_j,\tilde{a}_j,f)|}{j q_j} \int_{0}^{\infty} \exp \big( {-j f(u) X^{-1}} \big)  du  \nonumber\\
&  \leq \sum_{j=1}^X \frac{|S(q_j,\tilde{a}_j,f)|}{j q_j} \int_{0}^{\infty} \exp \big( {-j a_d u^d X^{-1}} \big)  du  \nonumber \\
 & \leq  \sum_{j=1}^J \frac{|S(q_j,\tilde{a}_j,f)|}{j^{1+\frac{1}{d}} q_j} \Gamma \Big(1+\frac{1}{d} \Big) \Big( \frac{X}{a_d} \Big)^{1/d} \label{second}.
\end{align}
Combining \eqref{tailexp}, \eqref{first} and \eqref{second} completes the proof.
\end{proof}

\begin{lemma} \label{constant}
Let $f \in \mathbb{Z}[y]$ be as in Theorem \ref{theorem1}. Then there exists a constant $0<C_f<1$ such that for all $q>1$ and $a \in \Z$ with $(a,q)=1$, 
\begin{equation*}
|S(q,a,f)| \leq C_f q.
\end{equation*}
\end{lemma}
\begin{proof}
Applying Weyl's inequality \cite[Theorem~4.3]{NA}, we have \eqref{wbound} for $b=0$. Thus for $0<\tilde{C}_{f}<1$ there exists an integer $R_f$ such that for all $q \geq R_f$ we have 
\begin{equation*}
|S(q,a,f)| \leq \tilde{C}_f q.
\end{equation*}
Now consider the case $q \leq R_f$. Now
\begin{equation*}
f(y) \equiv c \modu{p}
\end{equation*}
has at most $d \leq p-1$ solutions by Lagrange's Theorem \cite[Theorem~5.21]{TA} and the hypothesis $(a_d,\ldots,a_0)=1$. Thus $f$ can't be constant as a function modulo each prime $p$ for all $p>d$. Together with the hypothesis that $f$ is non-constant as a function modulo primes $p \leq d$, we can conclude by the Chinese Remainder Theorem that for each $q$ there exists at least one $1 \leq g:=g(q) \leq q$ such that $f(g) \not \equiv 0 \modu{q}$. Thus $af(g) \not \equiv 0 \modu{q}$ since $(a,q)=1$. Thus
\begin{equation*}
|S(f,a,q) | \leq \Big |q-1+e \Big( \frac{af(g)}{q} \Big) \Big | \leq \Big |q-1+e \Big( \frac{1}{R_f} \Big) \Big | \leq \hat{C}_f q,
\end{equation*}
for some $0<\hat{C}_f<1$. Taking $C_{f}:=\max\{ \tilde{C}_f, \hat{C}_f\}$ is sufficient.
\end{proof}

\section{Proof of Main Theorems} \label{maintheorems}
We will prove the main theorem using the Hardy--Littlewood circle method. From Cauchy's theorem we have, 
\begin{equation} \label{partit}
p_f(n)=\int_{0}^1 \rho^{-n} \exp \big( \Phi_f(\rho e(\Theta))-2 \pi i n \Theta \big) d \Theta.
\end{equation}
Recall $X$ is implicitly defined by \eqref{X}. Observing the periodicity in the integrand of \eqref{partit} with respect to $\Theta$, we may replace $[0,1]$ with $\mathcal{U}:=[-X^{\frac{1}{d}-1},1-X^{\frac{1}{d}-1}]$. For $a,q \in \mathbb{N}$ such that $1 \leq q \leq X^{\frac{1}{d}}$, $1 \leq a \leq q$ and $(a,q)=1$, we define each disjoint major arc
\begin{equation*}
\mathfrak{M}(q,a):=\Big \{ \Theta \in \mathcal{U}: \Big | \Theta-\frac{a}{q} \Big | \leq q^{-1} X^{\frac{1}{d}-1} \Big \}.
\end{equation*}
Furthermore we set
\begin{equation*}
\mathfrak{M}:= \bigcup_{ \substack{1 \leq a \leq q \leq X^{1/d} \\ (a,q)=1 }} \mathfrak{M}(q,a).
\end{equation*}
 
The minor arcs are defined by $\mathfrak{m}:=\mathcal{U} \setminus \mathfrak{M}$. In typical applications of the circle method, the main terms for the asymptotic in question usually consist of the contributions coming from all major arcs, and the error from the minor arcs. The problem of determining the asymptotic behaviour of $p_{\mathcal{A}_f}(n)$ defies this rule of thumb. 

By periodicity of the integrand of \eqref{partit} we can denote $\mathfrak{M}(1,1)$ as $\mathfrak{M}(1,0)$. The main contributions come from $\mathfrak{M}(1,0)$, whereas the contributions from $\mathfrak{M} \setminus \mathfrak{M}(1,0)$ will be subsumed by the error contributed by $\mathfrak{m}$. Thus the arcs in $\mathfrak{M} \setminus \mathfrak{M}(1,0)$ are referred to as \emph{auxiliary major arcs}. The anatomy of the proof follows the decomposition
\begin{equation*}
\bigg(\int_{\mathfrak{M}(1,0)}+\int_{\mathfrak{M} \setminus \mathfrak{M}(1,0)}+\int_{\mathfrak{m}} \bigg) \rho^{-n} \exp \big( \Phi_f(\rho e(\Theta))-2 \pi i n \Theta \big) d \Theta.
\end{equation*}
 
We will follow the implementation of the circle method outlined in \cite{BMZ,G,V2}, with the key difference being the deployment of our Lemmas in Section \ref{Auxiliary lemmas} and our treatment of the major arc $\mathfrak{M}(1,0)$. First we need the following lemma to handle the minor arcs and a subset of the major arc $\mathfrak{M}(1,0)$ that does not contribute to the main terms.
\begin{lemma} \label{minor}
Let $f \in \mathbb{Z}[y]$ be as in Theorem \ref{theorem1} and let $\Theta \in \mathfrak{m}$. Then for any $\varepsilon>0$, we have 
\begin{equation} \label{minores}
|\Phi_f \big(\rho e(\Theta) \big)| \ll_{f,\varepsilon} X^{\frac{1}{d}+\varepsilon-\frac{1}{d 2^{d-1}}}.
\end{equation}
Let $0<L<1$ be fixed and suppose $\Theta \in \mathfrak{M}(1,0) \setminus (-X^{\frac{L}{d}-1},X^{\frac{L}{d}-1})$. Then
\begin{equation} \label{auxM01}
|\Phi_f \big(\rho e(\Theta) \big)| \ll_{f,\varepsilon} X^{\frac{1}{d}+\varepsilon-\frac{L}{d 2^{d-1}}}.
\end{equation}
\end{lemma}
\begin{proof}
Let $J$ be a real positive number. It follows from  \eqref{tail} and \eqref{tailexp} that  
\begin{align}
|\Phi_f \big(\rho e(\Theta) \big)| & \ll_f \sum_{j=1}^J \frac{1}{j} \int_0^\infty f'(u)jX^{-1} \exp(-j f(u)/X) \Bigg | \sum_{n \le u}  \e ( j f(n) \Theta ) \Bigg| du+\Big(\frac{X}{J} \Big)^{\frac{1}{d}} \nonumber \\
& \ll_f \sum_{j=1}^J \frac{1}{j} \Big( \int_0^1+\int_1^{\infty} \Big) f'(u)jX^{-1} \exp(-j f(u)/X) \Bigg | \sum_{n \le u}  \e ( j f(n) \Theta ) \Bigg| du+\Big(\frac{X}{J} \Big)^{\frac{1}{d}} \nonumber \\
& \ll_f \sum_{j=1}^J \int_1^{\infty} d u^{d-1} jX^{-1} \exp(-j a_d u^d/X) \Bigg | \sum_{n \le u}  \e ( j f(n) \Theta ) \Bigg| du+\Big(\frac{X}{J} \Big)^{\frac{1}{d}}+ \frac{J}{X} \nonumber \\
& \ll_f \sum_{j=1}^J \int_0^{\infty} d u^{d-1} jX^{-1} \exp(-j a_d u^d/X) \Bigg | \sum_{n \le u}  \e ( j f(n) \Theta ) \Bigg| du+ \Big(\frac{X}{J} \Big)^{\frac{1}{d}}+\frac{J}{X} \label{phibd},
\end{align}
where we use the fact that all of the coefficients of the polynomial are non-negative on the third line of the above. For each $j$, we use Dirichlet's approximation theorem to choose $r_j \in \Z_{\geq 0}$ and $q_j \in \mathbb{N}$ such that
\begin{equation*}
\Bigg| j a_d \Theta-\frac{r_j}{q_j}   \Bigg| \leq q_j^{-1} X^{\frac{1}{d}-1} \quad \text{and} \quad 	q_j \leq X^{1-\frac{1}{d}}.
\end{equation*}
Applying Weyl's inequality \cite[Theorem~4.3]{NA} we have 
\begin{equation} \label{weylint}
\Bigg | \sum_{n \le u}  \e ( j f(n) \Theta ) \Bigg| \ll_{f,\varepsilon} u^{1+\varepsilon-2^{1-d}}+u^{1+\varepsilon} q_j^{-2^{1-d}}+u^{1+\varepsilon} \Big( \frac{q_j}{u^d} \Big)^{2^{1-d}}.
\end{equation}
Note that for any $\lambda>0$, integration by parts yields 
\begin{equation} \label{lambda}
\int_{0}^{\infty} u^{\lambda} \Big(jd u ^{d-1} X^{-1} \exp \big({-j a_d u^d/X} \big)  \Big) du \ll_{f} \Big(\frac{X}{j} \Big)^{\frac{\lambda}{d}}.
\end{equation}
Since $\Theta \notin \mathfrak{M}$ we must have $j a_d q_j \geq X^{\frac{1}{d}}$. We choose $J=X$. Thus \eqref{phibd}, \eqref{weylint}, \eqref{lambda} and the fact that $q_j \leq X^{1-\frac{1}{d}}$ imply that 
\begin{align*}
|\Phi_f \big(\rho e(\Theta) \big)| & \ll_{f,\varepsilon}1+ \sum_{j=1}^J \frac{1}{j} \Bigg(  \Big(\frac{X}{j} \Big)^{\frac{1+\varepsilon}{d}-\frac{1}{d2^{d-1}}}+\Big(\frac{X}{j} \Big)^{\frac{1+\varepsilon}{d}} q_j^{-2^{1-d}}+\Big(\frac{X}{j} \Big)^{\frac{1+\varepsilon}{d}-\frac{1}{2^{d-1}}} q_j^{2^{1-d}}    \Bigg)   \\
& \ll_{f,\varepsilon}1+ X^{\frac{1+\varepsilon}{d}-\frac{1}{d 2^{d-1}}} \sum_{j=1}^J  \big( j^{-1-\frac{1+\varepsilon}{d}+\frac{1}{d2^{d-1}} }+ j^{-1-\frac{1+\varepsilon}{d}+\frac{1}{2^{d-1}} } \big) \\
& \ll_{f,\varepsilon} X^{\frac{1+\varepsilon}{d}-\frac{1}{d 2^{d-1}}}.
\end{align*}
This proves \eqref{minores}. 
 
We now prove \eqref{auxM01}. Since $\Theta \in \mathcal{M}(1,0) \setminus (-X^{\frac{L}{d}-1},X^{\frac{L}{d}-1})$, $\Theta$ does not lie on any other major arc because they are disjoint. In particular, if $\Theta$ satisfies 
\begin{equation*}
\Bigg | \Theta-\frac{r}{q} \Bigg | \leq q^{-1} X^{\frac{L}{d}-1},
\end{equation*}
with $r \in \mathbb{Z}_{\geq 0}$, $q \in \mathbb{N}$, then $q>X^{\frac{1}{d}}$. Otherwise $\Theta$ would lie on a major arc or in the interval $(-X^{\frac{L}{d}-1},X^{\frac{L}{d}-1})$. Thus for each $j$, we use we use Dirichlet's approximation theorem to choose $r_j \in \Z_{\geq 0}$ and $q_j \in \mathbb{N}$ such that
\begin{equation*}
\Bigg| j a_d \Theta-\frac{r_j}{q_j}   \Bigg| \leq q_j^{-1} X^{\frac{L}{d}-1} \quad \text{and} \quad 	q_j \leq X^{1-\frac{L}{d}}.
\end{equation*}
By the above argument we see that $j q_j a_d>X^{\frac{1}{d}}$. Repeating the above computation with the new bound $q_j \leq X^{1-\frac{L}{d}}$ we obtain \eqref{auxM01}.

\end{proof}
Now we can begin the proof of Theorem \ref{theorem1}.
\begin{proof}
Fix $R<1$. Fix $L$ sufficiently small depending on $R, \boldsymbol{\alpha}$ and $\varepsilon$ as in Lemma \ref{1}. We first examine the major arc $\mathfrak{M}(1,0)$. Recall that
\begin{equation*}
\Psi _f(z) = \exp \big({\Phi _f}(z) \big).
\end{equation*}
Applying Lemma \ref{1}, we have the following for all $|\Theta| \leq X^{\frac{L}{d}-1}$:
\begin{equation} \label{initial}
\rho ^{ - n} \Psi _f(\rho e(\Theta )) = {\rho ^{ - n}}\exp \big({\Xi _f} \big( \rho e(\Theta ) \big) \big(1 + O_{f,R,\varepsilon}(X^{-R+\varepsilon}) \big),
\end{equation}
where  $\Xi _f(\rho e(\Theta ))$ is equal to the right hand side \eqref{auxphi} or \eqref{auxphi2} without the error term. Also note that we have
\begin{equation*}
\frac{X}{{1 - 2\pi i X\Theta }} = X\Delta {e^{i\phi }}
\end{equation*}
where $\phi = \arg(1+2 \pi i X \Theta)$. Thus $0 < |\phi|  \leq \pi/2$, so $0 < \cos(\phi/d)<1$. This means
\begin{equation} \label{impmod}
\bigg| {{{\bigg( {\frac{X}{{1 - 2\pi i X \Theta }}} \bigg)}^{\frac{1}{d}}}} \bigg| = {(X\Delta )^{\frac{1}{d}}}.
\end{equation}
We now study \eqref{partit} over the interval $\big[-\frac{3}{8 \pi X},\frac{3}{8 \pi X} \big]$. In other words, the main term of \eqref{initial}:
\begin{equation} \label{partit2}
\int_{-\frac{3}{8\pi X}}^{\frac{3}{8\pi X}} \rho^{-n} \exp \Big(\Xi_f \big(\rho e ( \Theta ) \big) - 2\pi i n \Theta \Big) d\Theta .
\end{equation}
The function
\begin{equation*}
K(\Theta):=\Xi _f  \big(\rho e(\Theta ) \big) - 2\pi i n \Theta 
\end{equation*}
over the region of integration can be Taylor expanded about $\Theta=0$ as
\begin{equation*}
 K(\Theta ) = \mathcal{C}+\zeta(0,\boldsymbol{\alpha}) \log \frac{X}{{{a_d}}}-Y{(2\pi X\Theta )^2} + G(\Theta ) 
\end{equation*}
where
\begin{equation*}
G(\Theta ) = \sum_{j = 3}^\infty  (u_j Y+v_j ) (2\pi iX\Theta )^j,
\end{equation*}
and
\begin{align*}
u_j&:={j-1+\frac{1}{d} \choose j} {1+\frac{1}{d} \choose 2}^{-1} \\
v_j&:=\zeta(0,\boldsymbol{\alpha}) \Big( \frac{1}{j}-\frac{u_j}{2} \Big)+O_f \Big( \frac{1}{X^{\frac{1}{d}}} \Big).  
\end{align*}

Since we have $R<1$ fixed, the last double summation in both \eqref{W1} and \eqref{W2} vanishes. Since all of the $\alpha_j \in \mathbb{R}_{>0}$ for $1 \leq j \leq d-1$, the induction argument establishing formulae for the residues $c_m$ on \cite[pp.~247--248]{MW} shows that $c_m \in \mathbb{R}$ for all $m \in \mathbb{N}$. Thus $Y,u_j,v_j \in \mathbb{R}$ for all $j$.  

Thus the integrand \eqref{partit2} becomes
\begin{equation} \label{parity}
  \bigg( {\frac{X}{{{a_d}}}} \bigg)^{\zeta (0,\boldsymbol{\alpha} )} \exp \Big( \mathcal{C}+\frac{n}{X} \Big) \int_{ - \frac{3}{8\pi X}}^{\frac{3}{8\pi X}} \exp \Big ( - Y{{(2\pi X\Theta )}^2} + G(\Theta ) \Big) d \Theta. 
\end{equation}
Considering the integral in \eqref{parity}, it will be convenient to use a parity argument
\begin{multline} \label{evenl}
\int_0^{\frac{3}{8\pi X}} \Big(\exp \big(G(\Theta) \big)+\exp \big(G(-\Theta) \big) \Big) \exp \Big(-Y(2\pi X\Theta)^2 \Big) d \Theta  \\
= \real \bigg( \int_0^{\frac{3}{8\pi X}} 2 \exp \big(G(\Theta)-Y(2\pi X \Theta)^2 \big) d\Theta \bigg).
\end{multline}
Making the change of variable $\phi = (2\pi X \Theta)^2 Y$ together with \eqref{evenl} allows us to re-write the integral on the right-hand side of \eqref{parity} as follows
\begin{equation} \label{upd}
 \frac{1}{2 \pi a_d^{\zeta(0,\boldsymbol{\alpha})}} \frac{\exp \Big( \mathcal{C}+\frac{n}{X} \Big)}{X^{1-\zeta(0,\boldsymbol{\alpha})} Y^{\frac{1}{2}}} \int_0^{9Y/16} \operatorname{Re} \Big(\exp \big( - \phi  + H(\phi ) \big) \Big){\phi ^{ - 1/2}} d\phi  
\end{equation}
where
\begin{equation*}
H(\phi):= \sum_{j=3}^\infty i^j (u_j+v_jY^{-1})\phi^{j/2}Y^{1-j/2}.
\end{equation*}
For a fixed $J$, we further decompose $H$ into  
\begin{equation*}
H(\phi):=H_J(\phi)+ \sum_{j=2J+3}^{\infty} i^j (u_j+v_jY^{-1})\phi^{j/2}Y^{1-j/2}.
\end{equation*}

Recalling the facts
\begin{equation*}
0 \leq \frac{a_{d-1}}{a_d} \leq \frac{d}{2} \quad \text{and} \quad \zeta(0,\boldsymbol{\alpha})=\zeta(0)-\frac{1}{d} \sum_{j=1}^{d-1} \alpha_j=-\frac{1}{2}-\frac{a_{d-1}}{d a_d},
\end{equation*}
we have 
\begin{equation} \label{zeta0}
|\zeta(0,\boldsymbol{\alpha})|<1.
\end{equation}
Thus 
\begin{equation*}
\frac{v_j}{u_j}=\zeta(0,\boldsymbol{\alpha}) \Big(\frac{1}{u_j j}-\frac{1}{2} \Big)+O_f \Big( \frac{1}{X^{\frac{1}{d}}} \Big).
\end{equation*}
Combining \eqref{zeta0} with the observation $j u_j \geq 1$ for any $j \geq 3$, we have 
\begin{equation*}
\Big |\zeta(0,\boldsymbol{\alpha}) \Big( \frac{1}{u_j j}-\frac{1}{2} \Big)  \Big | \leq \frac{1}{2}.
\end{equation*}
Thus for all $X$ sufficiently large we have $|v_j| \leq |u_j|$ for all $j \geq 3$. Also note that for $j \geq 2$ we have 
\begin{align*}
|v_{2j}| \leq u_{2j} \leq u_4=\frac{6d^2+5d+1}{12d^2}.
\end{align*}
Thus for $0 \leq \phi \leq 9 Y/16$, $Y$ sufficiently large, and $d \geq 2$, then 
\begin{align}
\text{Re} H_J(\phi)&=\text{Re} \sum_{j=3}^{2J+2} i^j (u_j+v_j Y^{-1}) \phi^{j/2} Y^{1-j/2}= \sum_{j=2}^{J+1} (-1)^j (u_{2j}Y+v_{2j}) \phi^j Y^{-j} \nonumber \\
& \leq a_4 (Y+1) \sum_{j=2}^{2J+2} \Big(\frac{\phi}{Y} \Big)^j \leq a_4 (Y+1) \sum_{j=2}^{\infty} \Big(\frac{\phi}{Y} \Big)^j \nonumber  \\
&=\frac{6d^2+5d+1}{12d^2} (Y+1) \frac{(\phi/Y)^2}{1-\phi/Y}  \nonumber \\
&\leq \frac{9}{7} \times \frac{6d^2+5d+1}{12d^2} (1+Y^{-1}) \phi \nonumber \\
&=\Big(\frac{18}{28}+\frac{15}{28 d}+\frac{3}{28d^2} \Big) (1+Y^{-1}) \phi \nonumber \\
&\leq \frac{105}{112} \Big(1+\frac{1}{105} \Big) \phi \nonumber \\
&<\Big(1-\frac{1}{2016} \Big) \phi \label{HJ},
\end{align}
where $Y>105$ is assumed.  Also notice by definition that $H_J(\phi) \rightarrow H(\phi)$ pointwise as $J \rightarrow \infty$. Thus for a parameter $Z>0$, we have 
\begin{equation*}
\int_Z^{9Y/16} {\operatorname{Re} \Big(\exp \big( - \phi  + H(\phi ) \big) \Big){\phi ^{ - 1/2}} d\phi }  \ll {Z^{ - 1/2}}{e^{ - Z/2016}}.
\end{equation*}
Taking $Z=2016 J \log Y$ yields
\begin{equation*}
\int_Z^{9Y/16} {\operatorname{Re} \Big(\exp  \big({ - \phi  + H(\phi )} \big)  \Big){\phi ^{ - 1/2}} d\phi }  \ll {Y^{ - J}}.
\end{equation*}
Consequently \eqref{upd} becomes
\begin{equation} \label{integral3}
 \frac{1}{2 \pi a_d^{\zeta(0,\boldsymbol{\alpha})}} \frac{\exp \Big( \mathcal{C}+\frac{n}{X} \Big)}{X^{1-\zeta(0,\boldsymbol{\alpha})} Y^{\frac{1}{2}}}  \Bigg( {\int_0^Z {\operatorname{Re} \Big(\exp \big( {- \phi  + H(\phi )} \big) \Big){\phi ^{ - 1/2}} d\phi }  + O({Y^{ - J}})} \Bigg). 
\end{equation}

When $0 \leq \phi \leq Z$, summing up a geometric series yields the estimate
\begin{equation*}
\sum_{j=2J+3}^{\infty} i^j (u_j+v_j Y^{-1}) \phi^{j/2} Y^{1-j/2} \ll \frac{\phi^{J+3/2} Y^{-1/2-J} }{1-(\phi/Y)^{1/2}} \ll \phi^{J+3/2} Y^{-1/2-J}
\end{equation*}
and so 
\begin{equation*}
\exp \Bigg( \sum_{j=2J+3}^{\infty} i^j (u_j+v_j Y^{-1}) \phi^{j/2} Y^{1-j/2} \Bigg)=1+O \big(\phi^{J+3/2} Y^{-1/2-J} \big).
\end{equation*}
Hence we now focus to the truncated sum $H_J(\phi)$. The integral occurring in \eqref{integral3} becomes
\begin{equation} \label{integral4}
\int_{0}^Z \Re \Big( \exp \big( -\phi+H_J(\phi) \big) \Big) \Big(1+O \big( \phi^{J+3/2} Y^{-1/2-J} \big) \Big) \phi^{-1/2} d \phi. 
\end{equation}
Since \eqref{HJ} holds, the error term in \eqref{integral4} contributes $O(Y^{-1/2-J})$. Thus \eqref{integral3} becomes
\begin{equation*}
 \frac{1}{2 \pi a_d^{\zeta(0,\boldsymbol{\alpha})}}  \frac{\exp \Big(\mathcal{C}+\frac{n}{X} \Big) }{X^{1-\zeta(0,\boldsymbol{\alpha})} Y^{\frac{1}{2}}} \Bigg( \int_0^Z \operatorname{Re} \Big( \exp \big( - \phi  + H_J(\phi ) \big) \Big) \phi ^{ - 1/2} d \phi   + O(Y^{ - J}) \Bigg).
\end{equation*}

Now consider
\begin{equation*}
\exp \big( H_J(\phi) \big) = \sum_{j=0}^\infty \frac{H_J(\phi)^j}{j!}.
\end{equation*}
For $0 \leq \phi \leq Z$ we have 
\begin{equation*}
H_J (\phi)=\sum_{j=3}^{2J+2} i^j(u_j+v_jY^{-1}) \phi^{j/2} Y^{1-j/2} \ll \sum_{j=3}^{\infty} Y (\phi/Y)^{\frac{j}{2}} \ll Y^{-\frac{1}{2}} \phi^{\frac{3}{2}} \leq Y^{-\frac{1}{4}}.
\end{equation*}
Thus 
\begin{equation*}
\Bigg | \int_{0}^Z \exp(-\phi) \text{Re} \big(H_J(\phi) \big)^j \phi^{-\frac{1}{2}} d \phi \Bigg| \leq Y^{-\frac{j}{4}},
\end{equation*}
and hence
\begin{equation*}
\sum_{j=4J+4}^{\infty} \int_0^Z \exp({-\phi}) \Re \Big( \frac{H_J(\phi)^j}{j!} \Big) \phi^{-1/2} d \phi \ll Y^{-J}.
\end{equation*}
Consequently we need to estimate
\begin{equation*}
\int_0^Z \exp(-\phi) \sum_{j=0}^{4J+3} \Re \Big( \frac{H_J(\phi)^j}{j!} \Big) \phi^{-1/2}.
\end{equation*}
Since $c_j,d_j,Y \in \mathbb{R}$ and can write 
\begin{equation*}
H_J(\phi)=\sum_{j=3}^{2J+2} \Big(u_j \big(Y^{-\frac{1}{2}} \big)^{j-2} +v_j \big(Y^{-\frac{1}{2}} \big)^{j} \Big)(i \phi^{1/2})^j,
\end{equation*}
we see that $H_J(\phi)$ can be viewed as a real polynomial in $i \phi^{1/2}$ of degree $2J+2$ for a fixed $Y$. Therefore 
\begin{equation*}
\sum_{j=0}^{4J+3} \frac{1}{j!} \big( H_j(\phi) \big)^j= \sum_{h=0}^{(2J+2)(4J+3)} p_h (Y^{-1/2})(i\phi^{1/2})^h
\end{equation*}
where the coefficients $p_h(z)$ are polynomials in $z$ of degree at most $h$. It can be verified that 
\begin{equation*}
p_0(\phi)=1, \quad p_1(\phi)=p_2(\phi)=0, \quad \text{and} \quad p_h(0)=0,
\end{equation*}
and for $h \geq 3$ the polynomial $p_h$ is even (resp. odd) when $h$ is even (resp. odd). Also note that  
\begin{equation*}
\int_{Z}^{\infty} \exp(-\phi) \phi^{\frac{h-1}{2}} d \phi \leq Y^{-J} \int_{0}^{\infty} \exp({-\phi}) \phi^{\frac{h-1}{2}} d \phi \ll Y^{-J}.
\end{equation*}
Therefore 
\begin{equation*}
\int_{-\frac{3}{8 \pi X}}^{\frac{3}{8 \pi X}} \rho^{ - n} \exp \Big( \Xi _f (\rho e(\Theta )) - 2 \pi i n \Theta \Big) d \Theta   \\
 =  \frac{1}{2 \pi a_d^{\zeta(0,\boldsymbol{\alpha})}}  \frac{\exp \Big(\mathcal{C}+\frac{n}{X} \Big)}{X^{1-\zeta(0,\boldsymbol{\alpha})} Y^{\frac{1}{2}}} \big(I + O(Y^{ - J}) \big)
\end{equation*}
where
\begin{align}
I&= \int_{0}^Z \Re \big( \exp (-\phi+H_J(\phi) ) \big) \phi^{-1/2} d \phi \nonumber  \\
&=\sum_{\substack{h=0 \\ h \text{ even} }}^L p_{h}(Y^{-1/2})  \int_{0}^Z \exp(-\phi) \phi^{\frac{h-1}{2}} d \phi +O(Y^{-J}) \nonumber \\
&= \Gamma \Big( \frac{1}{2} \Big) + \sum_{h = 2}^{L/2} p_{2h} (Y^{-1/2}) \Gamma \Big( h + \frac{1}{2} \Big) + O(Y^{ - J}) \label{polyy}.
\end{align}
Note that $p_{2h}$ is an even polynomial, so \eqref{polyy} is indeed a polynomial in the variable $Y^{-1}$. So let $w_q$ be the coefficient of $Y^{-q}$ in \eqref{polyy}. Thus 
\begin{multline}  \label{mainc}
\int_{ - \frac{3}{8\pi X}}^{\frac{3}{8 \pi X}} \rho^{ - n} \exp ( \Xi _f (\rho e(\Theta )) - 2 \pi i n \Theta ) d \Theta \\
=\frac{1}{2 \pi a_d^{\zeta(0,\boldsymbol{\alpha})}}  \frac{\exp \Big(\mathcal{C}+\frac{n}{X} \Big)}{X^{1-\zeta(0,\boldsymbol{\alpha})} Y^{\frac{1}{2}}} \Big(\pi^{1/2}+\sum_{q=1}^{J-1} w_q Y^{-q}+O(Y^{ - J}) \Big). 
\end{multline}
Recalling \eqref{initial} we have 
\begin{multline} \label{mainc2}
\int_{ - \frac{3}{8\pi X}}^{\frac{3}{8 \pi X}} \rho^{ - n} \exp ( \Phi _f (\rho e(\Theta )) - 2 \pi i n \Theta ) d \Theta \\
=\frac{1}{2 \pi a_d^{\zeta(0,\boldsymbol{\alpha})}}  \frac{\exp \Big(\mathcal{C}+\frac{n}{X} \Big)}{X^{1-\zeta(0,\boldsymbol{\alpha})} Y^{\frac{1}{2}}} \Big(\pi^{1/2}+\sum_{q=1}^{J-1} w_q Y^{-q}+O(Y^{ - J})+O_{f,R,\varepsilon} \big( X^{-R+\varepsilon} \big) \Big).
\end{multline}

We now need to consider \eqref{partit} over $\mathcal{U} \setminus [-\frac{3}{8 \pi X},\frac{3}{8 \pi X}]$. It suffices to prove
\begin{equation} \label{remain} 
 \int_{\mathcal{U} \setminus [ - \frac{3}{8 \pi X},\frac{3}{8 \pi X}]} \rho ^{- n} \exp \Big (\Phi_f \big(\rho e(\Theta) \big) - 2\pi in\Theta \Big) d\Theta \ll \frac{1}{2 \pi a_d^{\zeta(0,\boldsymbol{\alpha})}}  \frac{\exp \Big(\mathcal{C}+\frac{n}{X} \Big)}{X^{1-\zeta(0,\boldsymbol{\alpha})} Y^{\frac{1}{2}}} Y^{ - J}.
\end{equation}
We treat the left side of \eqref{remain} over $\mathcal{U} \setminus \big[-\frac{3}{8 \pi X}, \frac{3}{8 \pi X} \big]$ in three stages.  

On the set $(-X^{\frac{L}{d}-1},X^{\frac{L}{d}-1} ) \setminus [-3/(8 \pi X),3/(8 \pi X)]$, we have 
\begin{equation*}
\Delta \leq \frac{4}{5}. 
\end{equation*}
Fix $\eta>0$ small such that $(1+\eta) (4/5)^{\frac{1}{d}}<1$. Applying Lemma \ref{1} and recalling \eqref{impmod} then we see that
\begin{equation*}
  \big| {\Phi _f}(\rho e(\Theta )) \big | \leq (1+\eta) \Big(\frac{4}{5} \Big)^{\frac{1}{d}} \frac{a_d^{-1/d}}{d} \Gamma \Big( \frac{1}{d} \Big) \zeta \Big( {1 + \frac{1}{d}} \Big) X^{\frac{1}{d}},
   \end{equation*}
for sufficiently large $X$. This implies \eqref{remain} over the set $(-X^{\frac{L}{d}-1},X^{\frac{L}{d}-1} ) \setminus [-3/(8 \pi X),3/(8 \pi X)]$.

Applying Lemma \ref{minor}, it is clear that \eqref{remain} is satisfied when the integral is restricted to the domain  $\Big(\mathfrak{M}(1,0) \setminus (-X^{\frac{L}{d}-1},X^{\frac{L}{d}-1} )\Big) \cup \mathfrak{m}$.

The third stage is to study the integral in \eqref{remain} on the auxiliary major arcs $\mathfrak{M}(q,a)$. Hence we have $1<q \leq X^{\frac{1}{d}}$ and $\beta=\Theta-a/q$ satisfies $|\beta| \leq q^{-1} X^{\frac{1}{d}-1}$. 
Thus by Lemma \ref{lemma3} we have 
\begin{equation} \label{lemma3eq}
\big | \Phi_f \big( \rho e(\Theta) \big) \big| \leq a_d^{-1/d} \Gamma \Big( 1 + \frac{1}{d} \Big) X^{\frac{1}{d}} \sum\limits_{j = 1}^\infty  \frac{|S(q_j,\tilde{a}_j,f)|}{j^{(d + 1)/d} q_j} +O_{\varepsilon} \big( X^{\frac{1}{d}-\frac{1}{d 2^{d-1}}+2 \varepsilon} \big)
\end{equation}
for all sufficiently large $X$. For each $q_j$, we need to bound the series 
\begin{equation*}
\sum_{j=1}^{\infty} \frac{|S(q_j,\tilde{a}_j,f)|}{j^{\frac{d+1}{d}} q_j}.
\end{equation*}
Recalling the notation from Lemma \ref{lemma3} we have 
\begin{equation*}
\tilde{a}_j=\frac{aj}{(j,q)} \quad \text{and} \quad q_j=\frac{q}{(j,q)}.
\end{equation*}

When $q \mid j$ then $q_j=1$ and hence $|S(q_j,\tilde{a}_j,f)|=1$. When $j \centernot\mid q$ we have $q_j>1$ and then applying Lemma \ref{constant} we have $|S(q_j,\tilde{a}_j,f)| \leq (1-\delta_f)q_j$ where $\delta_f>0$ is defined by $C_f:=1-\delta_f$. Thus   
\begin{align}
\sum_{j=1}^{\infty} \frac{|S(q_j,\tilde{a}_j,f)|}{j^{\frac{d+1}{d}} q_j} &= \sum_{\substack{j=1 \\ q \mid j}}^{\infty} \frac{1}{j^{1+\frac{1}{d}}}+\sum_{\substack{j=1 \\ j \nmid q}}^{\infty} \frac{1-\delta_f}{j^{1+\frac{1}{d}}} \nonumber \\
&=(1-\delta_f) \Big(1-\frac{1}{q^{1+\frac{1}{d}}} \Big) \zeta \Big(1+\frac{1}{d} \Big)+\frac{1}{q^{1+\frac{1}{d}}} \zeta \Big (1+\frac{1}{d} \Big) \nonumber \\
&=\Big(1-\delta_f+\frac{\delta_f}{q^{1+\frac{1}{d}}} \Big) \zeta \Big(1+\frac{1}{d} \Big) \nonumber \\
&<\Big(1-\frac{\delta_f}{2} \Big) \zeta \Big(1+\frac{1}{d} \Big). \label{const2}
\end{align}
Combining \eqref{lemma3eq} and \eqref{const2} we obtain
\begin{equation*}
\big |\Phi_f \big( \rho e(\Theta) \big) \big| \leq \Big(1-\frac{\delta_f}{2} \Big) \frac{a_d^{-1/d}}{d} \Gamma \Big(\frac{1}{d} \Big) \zeta \Big(1+\frac{1}{d} \Big) X^{\frac{1}{d}},
\end{equation*} 
and we see that \eqref{remain} holds on the auxiliary major arcs.

Recalling \eqref{mainc2} and \eqref{remain}, any choice $1<J<dR$ will mean the error term $O_{f,R,\varepsilon}(X^{-R+\varepsilon})$ can be absorbed.
\end{proof}

Here we offer some remarks on possible alternative strategies to the one adopted in this paper.
\begin{remark}
For the sake of this remark assume $a_0=0$ and $a_j>0$ for $1 \leq j \leq d$. Following the proof of \cite[Thm. C4]{MV}, Parseval's theorem in the form
\begin{equation*}
\int_0^\infty  f(t)g(t)t^{s - 1}dt  = \frac{1}{2\pi i}\int_{(c)} \mathcal{F}(w)\mathcal{G}(s - w) dw
\end{equation*}
where $c>0$ and $\mathcal{F}$ and $\mathcal{G}$ are the Mellin transforms of $f$ and $g$, respectively, one can show by induction that for $d \in \N$ we have
\begin{align} \label{parsevalinduction}
I_d(s) &:= \int_0^\infty  \exp \bigg( - jx\sum\limits_{r = 1}^d a_rt^r \bigg)t^{s - 1}dt \nonumber \\
       & = \frac{1}{d}\sum\limits_{n_1 = 0}^\infty  \cdots \sum\limits_{n_{d - 1} = 0}^\infty   \frac{( - 1)^{n_1 +  \cdots  + n_{d - 1}}}{n_1! \cdots n_{d - 1}!}a_1^{n_1} \cdots a_{d - 1}^{n_{d - 1}}a_d^{ - \tfrac{1}{d}(s + \sum\nolimits_{r = 1}^{d - 1} r n_r )} \nonumber \\
       &\quad \times {(jx)^{n_1 +  \cdots  + n_{d - 1} - \tfrac{1}{d}(s + \sum\nolimits_{r = 1}^{d - 1} r n_r )}}\Gamma \bigg( \frac{1}{d}\bigg( s + \sum\limits_{r = 1}^{d - 1} r n_r  \bigg) \bigg)  
\end{align}
for $\real(x)>0$ and $\real(s)>0$. If we let $x=1/X - 2\pi i\Theta$ and $s=1$, then the main term of Lemma 3 becomes $I_d(1)$.
\end{remark}

\begin{remark}
A temptation is to take a multivariable approach to Lemma \ref{1} by means of repeatedly applying the Cahen--Mellin transform. By this transform, we have
\begin{align} \label{specialcahen}
\exp (- j a_r n^r x ) = \frac{1}{2\pi i}\int_{(c_r)} \Gamma (s_r)(j a_r n^r x)^{-s_r} ds_r 
\end{align}
where $x=1/X - 2\pi i \Theta$ and $c_r>0$ for $r=1,\cdots,d$. This allows us to write
\begin{align} \label{dfoldintegrals}
\Phi_{f}(\rho e(\Theta )) &= \sum_{j = 1}^\infty  \sum_{n = 1}^\infty  \frac{1}{j} \exp \bigg(  - j\bigg( \sum_{r = 1}^d a_r n^r  \bigg)\bigg( \frac{1}{X} - 2\pi i\Theta  \bigg) \bigg) \nonumber \\
& = \sum_{j = 1}^\infty \sum_{n = 1}^\infty \frac{1}{j}  \prod_{r = 1}^d \bigg( \frac{1}{2\pi i}\int_{(c_r)} \Gamma ({s_r})(ja_r n^r x)^{-s_r}ds_r \bigg) \nonumber \\ 
& = \frac{1}{(2\pi i)^d}\int_{(c_1)} \cdots \int_{(c_d)}  \bigg( \prod\limits_{r = 1}^d \Gamma (s_r)  \bigg)\zeta \bigg( \sum\limits_{r = 1}^d s_r  + 1 \bigg) \nonumber \\
	 & \quad \times \zeta \bigg( \sum\limits_{r = 1}^d r s_r  \bigg)\bigg( \prod\limits_{r = 1}^d a_r^{ - s_r}  \bigg)x^{ - \sum\nolimits_{r = 1}^d s_r } ds_d \cdots ds_1  
\end{align}
provided that 
\begin{align} \label{conditionsonmultiples}
\sum_{r=1}^d \sigma_r > \delta \quad \textnormal{and} \quad \sum_{r=1}^d k_r \sigma_r > 1 + \delta
\end{align}
where $s_r=\sigma_r+it_r$ with $\sigma_r, t_r \in \R$ and $0 < \delta <1$ for $r=1,\cdots,d$. However, dealing with the main terms arising from computing the above $d$-fold integral is not as convenient as appealing to the Matsumoto--Weng zeta function and it is not entirely clear how to obtain satisfactory error terms.

\end{remark}

\section{Acknowledgments}
The authors wish to thank Professor Alexandru Zaharescu, Professor Ole Warnaar, Professor Wadim Zudilin and the anonymous referee for their helpful comments on the manuscript. The authors are grateful for the encouragement from Professor Bruce Berndt and Amita Malik. The first author is supported as a teaching assistant at the University of Illinois at Urbana--Champaign, and the second author is supported in the capacity of a J.L Doob Assistant Professor at the same institution.

\end{document}